\documentclass[a4paper,reqno,10pt]{amsart}
\makeatletter
\newcommand*{\rom}[1]{\expandafter\@slowromancap\romannumeral #1@}
\makeatother
\usepackage{epsf,graphicx,epsfig}
\usepackage{amscd}
\usepackage{amsmath,latexsym,amssymb,amsthm}
\usepackage[nospace,noadjust]{cite}
\usepackage{textcomp}
\usepackage{setspace,cite}
\usepackage{mathrsfs,amsmath}
\usepackage{chemarrow}
\usepackage{lscape,fancyhdr,fancybox}
\usepackage{stmaryrd}
\usepackage[all,cmtip]{xy}
\usepackage{tikz-cd}
\usepackage{cancel}
\usetikzlibrary{shapes,arrows,decorations.markings}
\usepackage{graphicx}

\usepackage{color}
\usepackage{bbm, dsfont}
\usepackage{xcolor}
\usepackage{url}
\usepackage{enumerate}
\usepackage[mathscr]{euscript}
  \theoremstyle{plain}
\swapnumbers
    \newtheorem{thm}{Theorem}[section]
    \newtheorem{proposition}[thm]{Proposition}
     
   \newtheorem{lemma}[thm]{Lemma}

    \newtheorem{subsec}[thm]{}
\theoremstyle{definition}
    \newtheorem{definition}[thm]{Definition}
        \newtheorem{remark}[thm]{Remark}
    \newtheorem{exam}[thm]{Example}

\theoremstyle{remark}

\title{}
\author{}
\date{}
\usepackage{amssymb}

\usepackage{hyperref}
\hypersetup{
colorlinks,
citecolor=blue,
filecolor=black,
linkcolor=blue,
urlcolor=black
}

\begin{document}

\title[Homotopy conformal algebras]{Homotopy conformal algebras}

\author{Anupam Sahoo}
\address{Department of Mathematics,
Indian Institute of Technology, Kharagpur 721302, West Bengal, India.}
\email{anupamsahoo23@gmail.com}

\author{Apurba Das}
\address{Department of Mathematics,
Indian Institute of Technology, Kharagpur 721302, West Bengal, India.}
\email{apurbadas348@gmail.com, apurbadas348@maths.iitkgp.ac.in}

\begin{abstract}
The notion of conformal algebras was introduced by Victor G. Kac using the axiomatic description of the operator product expansion of chiral fields in conformal field theory. The structure theory, representations and cohomology of Lie and associative conformal algebras are extensively studied in the literature. In this paper, we first introduce $A_\infty$-conformal algebras as the homotopy analogue of associative conformal algebras, provide some equivalent descriptions and prove the homotopy transfer theorem. We characterize some $A_\infty$-conformal algebras in terms of Hochschild cohomology classes of associative conformal algebras. Next, we introduce associative conformal $2$-algebras as the categorification of associative conformal algebras. We show that the category of associative conformal $2$-algebras and the category of $2$-term $A_\infty$-conformal algebras are equivalent. Finally, we consider $L_\infty$-conformal algebras and find their relations with $A_\infty$-conformal algebras.
\end{abstract}

\maketitle

%\curraddr{}
%\email{}

%\subjclass[2010]{}
%\keywords{}

\medskip

\begin{center}



\end{center}

%\medskip

%\noindent {\sf Date of resubmission:} July 26, 2021.

\thispagestyle{empty}

\tableofcontents

\section{Introduction}
\subsection{Conformal algebras} The concept of conformal algebras was first introduced by V. G. Kac as a useful tool to study vertex algebras \cite{kac-1}, \cite{kac-2}. Conformal algebras encode the axiomatic description of the operator product expansion of chiral fields in conformal field theory. They appear naturally in the study of formal distribution Lie algebras. They also have connections with infinite-dimensional Lie algebras satisfying the locality property. Hence the theory of finite conformal algebras is useful to provide a classification of some infinite-dimensional Lie algebras. Lie conformal algebras also have applications in two-dimensional conformal field theory and string theory \cite{kac-1}, \cite{sole-kac}. In the last twenty-five years,
the structure theory and representations of Lie conformal algebras and associative conformal algebras are extensively studied by several authors \cite{andrea}, \cite{fattori-kac}, \cite{retakh}, \cite{Kolesnikov}, \cite{roitman-1}, \cite{roitman-3}, \cite{hong-su}.

The notion of cohomology theory in the context of operator product expansion was first considered in \cite{kimura-voronov} for vertex algebras and conformal field theories. Subsequently, the cohomology theory of conformal algebras was systematically described by Bakalov, Kac and Voronov \cite{bakalov-kac-voronov}. The cohomology of Lie conformal algebras generalizes the Chevalley-Eilenberg cohomology of usual Lie algebras and the cohomology of associative conformal algebras generalizes the classical Hochschild cohomology of associative algebras. These cohomologies describe the deformation of respective conformal algebras. See \cite{sole-kac}, \cite{dolguntseva}, \cite{kolesnikov-kozlov}, \cite{hou} for more details about the cohomology of conformal algebras and their applications. We will follow the conventions of \cite{sole-kac} while considering the cohomology (although the cohomology is the same with \cite{bakalov-kac-voronov}).

\subsection{Homotopy algebras and 2-algebras} The first instance of homotopy algebras arises in the work of Stasheff in the recognization of topological loop spaces \cite{stasheff}. Since its inception, $A_\infty$-algebras (also called strongly homotopy associative algebras) have found important applications in homotopy theory, topology and geometry \cite{penkava-schwarz}, \cite{keller}. Among other flexibilities, $A_\infty$-algebras satisfy the homotopy transfer theorem whereas usual associative algebras do not. The concept of $L_\infty$-algebras (which are the Lie analogue of $A_\infty$-algebras) was introduced by  Lada, Markl and Stasheff \cite{lada-stasheff}, \cite{lada-markl}. They have close connections with deformation theory, higher differential geometry and mathematical physics \cite{fre-zam}, \cite{kontsevich-soibelman}, \cite{markl}. On the other hand, in \cite{baez} Baez and Crans introduced the notion of Lie $2$-algebras as the categorification of Lie algebras and find a relationship with the Zamolodchikov tetrahedron equation. They showed that the category of Lie $2$-algebras is equivalent to the category of $2$-term $L_\infty$-algebras. In the same article, the authors observed that skeletal $L_\infty$-algebras are characterized by the third cocycles of the Chevalley-Eilenberg complex of Lie algebras, and strict $L_\infty$-algebras are characterized by crossed modules of Lie algebras. These results are fundamental in connections with Lie algebras, $L_\infty$-algebras and Lie $2$-algebras. All these results are valid for algebras of other types. In particular, one can also consider $2$-term $A_\infty$-algebras and associative $2$-algebras and show their categorical equivalence. The standard skew-symmetrization functor from the category of associative algebras to the category of Lie algebras has a generalization to the context of homotopy algebras. That is, there is a commutative diagram 
\[
\begin{tikzcd}
\mathrm{associative~ algebras ~} \ar[r, hook] \ar[d, hook] &  2\text{-term }A_\infty\text{-algebras} \simeq \text{associative }2\text{-algebras}  \ar[r, hook] \ar[d, hook] & A_\infty\text{-algebras} \ar[d, hook]\\
\mathrm{Lie~ algebras}  \ar[r, hook] &   2\text{-term }L_\infty\text{-algebras} \simeq \text{Lie }2\text{-algebras 
  }
 \ar[r, hook]  & L_\infty\text{-algebras}.
\end{tikzcd}
\]

\medskip

\subsection{Bridges between conformal algebras, homotopy algebras and 2-algebras} Our aim in this paper is to find bridges between conformal algebras, homotopy algebras and $2$-algebras. To do so, we first introduce $A_\infty$-conformal algebras as the conformal analogue of $A_\infty$-algebras. We also define representations of an $A_\infty$-conformal algebra. Given a graded $\mathbb{C}[\partial]$-module $\mathcal{A}$, we construct a graded Lie algebra whose Maurer-Cartan elements correspond to $A_\infty$-conformal algebra structures on $\mathcal{A}$. This characterization allows us to define the cohomology of an $A_\infty$-conformal algebra with coefficients in a representation. Our cohomology of an $A_\infty$-conformal algebra generalizes the Hochschild cohomology of associative conformal algebras.
%We give some equivalent descriptions of such algebras including the Maurer-Cartan characterization of $A_\infty$-conformal algebras. Among others, we define representations of an $A_\infty$-conformal algebra and introduce cohomology with coefficients in a representation.
Next, we prove (various forms of) the homotopy transfer theorem for $A_\infty$-conformal algebras which shows that the category of $A_\infty$-conformal algebras is much flexible than the category of associative conformal algebras.
We also consider $2$-term $A_\infty$-conformal algebras and characterize some important classes of  $2$-term $A_\infty$-conformal algebras in terms of Hochschild cohomology of associative conformal algebras. We denote the category of $2$-term $A_\infty$-conformal algebras by ${\bf 2A_\infty conf}$. Next, we consider associative conformal $2$-algebras as the categorification of associative conformal algebras.  The collection of all associative conformal 2-algebras and homomorphisms between them forms a
category, denoted by {\bf conf2Alg}. We show that the category {\bf conf2Alg} of associative conformal $2$-algebras and the category ${\bf 2A_\infty conf}$ of $2$-term $A_\infty$-conformal algebras are equivalent. 

In the Lie context, it is worth mentioning that the concepts of $2$-term $L_\infty$-conformal algebras and Lie conformal $2$-algebras are recently considered by Zhang \cite{tao}. The author also showed that the category of $2$-term $L_\infty$-conformal algebras and the category of Lie conformal $2$-algebras are equivalent. However, he didn't consider the full concept of $L_\infty$-conformal algebras. In this paper, we also fill up this gap 
by introducing $L_\infty$-conformal algebras as the homotopy analogue of Lie conformal algebras. The notion of Zhang is then a particular case of our $L_\infty$-conformal algebras. We show that a suitable skew-symmetrization of an $A_\infty$-conformal algebra yields an $L_\infty$-conformal algebra. Hence we obtain the following diagram that generalizes the above diagram in the conformal context

%\[
%\xymatrix{
%\mathrm{ass~ algebras ~} \ar@{^{(}->}[r] & 2-term & 2-algebras \ar@{^{(}->}[r] & \mathrm{homotopy~ass~ algebras} \\
%\mathrm{Lie~ algebras} \ar@{^{(}->}[r] & 2-term & 2-algebras \ar@{^{(}->}[r] & \mathrm{homotopy~Lie~ algebras}
%}
%\]

%\[
%\xymatrix{
%\mathrm{ass ~conf.~ algebras} \ar@{^{(}->}[r] & 2-term & 2-algebras \ar@{^{(}->}[r] & \mathrm{homotopy~ass~conf.~ algebras} \\
%\mathrm{Lie~ conf.~ algebras} \ar@{^{(}->}[r] & 2-term & 2-algebras \ar@{^{(}->}[r] & \mathrm{homotopy~Lie~conf. algebras}
%}
%\]

%\[
%\begin{tikzcd}
%\mathrm{associative~conf.~ algebras ~} \ar[r, hook] \ar[d, hook] &   \substack{2-term \\ 2-algebras}  \ar[d, hook] \ar[r, hook]  & A_\infty\mathrm{-conformal ~algebras} \ar[d, hook]\\
%\mathrm{Lie~conf.~ algebras}  \ar[r, hook]  &  \substack{2-term \\ 2-algebras}  
% \ar[r, hook]   & L_\infty\mathrm{-conformal ~ algebras}
%\end{tikzcd}
%\]

\[
\begin{tikzcd}
\mathrm{associative~conf.~ algebras ~} \ar[r, hook] \ar[d, hook] &  \substack{2\text{-term }A_\infty\text{-conf. ~algebras} \\ \simeq ~ \text{associative~conf. }2\text{-algebras} }  \ar[r, hook] \ar[d, hook] & A_\infty\text{-conf.~algebras} \ar[d, hook]\\
\mathrm{Lie~conf.~ algebras}  \ar[r, hook] &  \substack{   2\text{-term }L_\infty\text{-conf.~algebras} \\ \simeq ~\text{Lie~conf. }2\text{-algebras 
  }}
 \ar[r, hook]  & L_\infty\text{-conf.~algebras}.
\end{tikzcd}
\]

\medskip

\subsection{Organization of the paper} The paper is organized as follows. 
In section \ref{section-2}, we recall associative conformal algebras and their Hochschild cohomology with coefficients in a conformal bimodule. In section \ref{section-3}, we introduce the notion of $A_\infty$-conformal algebras which are the homotopy analogue of associative conformal algebras. We also consider representations of an $A_\infty$-conformal algebra and introduce the cohomology of an $A_\infty$-conformal algebra with coefficients in a representation.
In section \ref{section-4}, we mainly consider $2$-term $A_\infty$-conformal algebras and classify some particular classes of such algebras. The notion of conformal associative $2$-algebras is introduced in section \ref{section-5}. We also show that the categories ${\bf 2A_\infty conf}$ and ${\bf conf2Alg}$ are equivalent. Finally, in section \ref{section-6}, we recall Lie conformal algebras and introduce $L_\infty$-conformal algebras, and show that a suitable skew-symmetrization of an $A_\infty$-conformal algebra yields an $L_\infty$-conformal algebra.

%whose underlying graded $\mathbb{C}[\partial]$-module $\mathcal{A}$ is concentrated in degrees $0$ and $1$. We call them $2$-term conformal $A_\infty$-algebra. The collection of all $2$-term conformal $A_\infty$-algebras and homomorphisms between them form a category, denoted by $2\mathrm{Conf}A_\infty$. A special case of $2$-term conformal $A_\infty$-algebras includes skeletal conformal $A_\infty$-algebras and strict conformal $A_\infty$-algebras. Motivated by the results of Baez and Crans, we show that skeletal conformal $A_\infty$-algebras are classified by the third cocycles in the Hochschild complex of associative conformal algebras. We introduce crossed modules of associative conformal algebras and show that strict conformal $A_\infty$-algebras are characterized by the cross modules.

%In section \ref{section-5}, we introduce conformal associative $2$-algebras as the categorification of associative conformal algebras. The collection of all associative conformal $2$-algebras and homomorphisms between them form a category, denoted by Conf$2$Alg. We show that the categories $2\mathrm{Conf}A_\infty$ and Conf$2$Alg are equivalent.

%Finally, in section \ref{section-6}, we recall conformal Lie algebras and introduce the concept of conformal $L_\infty$-algebras. We show that a suitable skew-symmetrization of conformal $A_\infty$-algebras gives rise to conformal $L_\infty$-algebras.

\subsection*{Notations}

All vector spaces, linear and multilinear maps, and tensor products are over the field $\mathbb{C}$ of complex numbers.

\section{Associative conformal algebras and their cohomology}\label{section-2}

In this section, we recall associative conformal algebras, conformal bimodules and the cohomology theory of associative conformal algebras. Our main referances are \cite{bakalov-kac-voronov}, \cite{dolguntseva}, \cite{kolesnikov-kozlov}.

\begin{definition}
An {\em associative conformal algebra} is a $\mathbb{C}[\partial]$-module $A$ equipped with a $\mathbb{C}$-linear map (called the $\lambda$-multiplication) $\cdot_\lambda {\cdot} : A\otimes A \to A[\lambda]$, $a\otimes b \mapsto a  _\lambda b$ satisfying the conformal sesquilinearity condition:
\begin{align}\label{confor-ses}
 (\partial a)_{\lambda} b = - \lambda a_\lambda b \quad \mathrm{and } \quad
  a_{\lambda} (\partial b) = (\partial + \lambda)a_{\lambda} b,
  \end{align}
and the following conformal associativity:  $a _{\lambda}(b_{\mu} c) =  (a_{\lambda} b)_{\lambda+ \mu} c$, for all $a, b, c \in A$.
\end{definition}
An associative conformal algebra as above may be simply denoted by $A$ when the $\lambda$-multiplication map is clear from the context.

Let $A$ be an associative conformal algebra with the $\lambda$-multiplication $\cdot_\lambda \cdot$. For each $j \geq 0$, we define the $j$-th product on $A$ by a $\mathbb{C}$-linear map $\cdot_{(j)} \cdot : A \otimes A \rightarrow A$ that satisfies
\begin{align}\label{lambda-j}
    a_\lambda b = \sum_{j \geq 0} \frac{\lambda^j}{j!} a_{(j)} b, \text{ for } a, b \in A.
\end{align}
Since $a_\lambda b$ is an $A$-valued polynomial in $\lambda$, there are finitely many $j$'s (that depends on $a$ and $b$ both) for which $a_{(j)} b \neq 0$. With the correspondence (\ref{lambda-j}) between the $\lambda$-multiplication and $j$-th products, the conformal sesquilinearity condition (\ref{confor-ses}) of the $\lambda$-multiplication is equivalent to the conditions
\begin{align*}
    (\partial a)_{(j)} b = - j a_{(j-1)} b, \quad a_{(j)} (\partial b) = \partial (a_{(j)} b) + j a_{(j-1)} b, \text{ for } a, b \in A \text{ and } j \geq 0.
\end{align*}
Moreover, the conformal associativity is equivalent to
\begin{align}\label{conf-ass-j}
    a_{(j)} (b_{(k)} c) = \sum_{p=0}^j \binom{j}{p} (a_{(p)} b)_{(j+k-p)} c, \text{ for all } a, b, c \in A \text{ and } j, k \geq 0.
\end{align}

Associative conformal algebras appear naturally in the context of formal distribution associative algebras \cite{bakalov-kac-voronov}. Another class of examples are current conformal algebras associated with associative algebras. More precisely, let $A$ be an associative algebra. Then the $\mathbb{C}[\partial]$-module $\mathrm{Cur } A = \mathbb{C}[\partial] \otimes A$ carries an associative conformal algebra structure with the $\lambda$-multiplication $a_\lambda b := a b$, for all $a, b \in A$.

%\begin{exam}
    Let $M$ be any $\mathbb{C}[\partial]$-module. A {\em conformal linear endomorphism} of $M$ is a $\mathbb{C}$-linear map $f: M \rightarrow M[\lambda],~ m \mapsto f_\lambda (m)$ that satisfies $f_\lambda (\partial m) = (\partial + \lambda) f_\lambda (m)$, for all $m \in M$. The space of all conformal linear endomorphisms of $M$ is denoted by $\mathrm{Cend}(M)$. Note that $\mathrm{Cend}(M)$ has canonical structure of a $\mathbb{C}[\partial]$-module given by $(\partial f)_\lambda (m) := - \lambda f_\lambda (m)$, for all $m \in M$. Then $\mathrm{Cend}(M)$ is an associative conformal algebra with the $\lambda$-multiplication
    \begin{align*}
        (f_\lambda g)_\mu (m) := f_\lambda (g_{\mu - \lambda} (m)), \text{ for } f, g \in \mathrm{Cend}(M), m \in M.
    \end{align*}
%\end{exam}

\begin{definition}
    Let $A$ be an associative conformal algebra. A {\em conformal $A$-bimodule} is a $\mathbb{C}[\partial]$-module $M$ equipped with two $\mathbb{C}$-linear maps (called the left and right $\lambda$-actions, respectively) $A \otimes M \to M[\lambda]$, $a \otimes m \mapsto a_{\lambda} m$ and $M \otimes A \to M[\lambda]$, $m \otimes a \mapsto m_{\lambda}a$ satisfying the conformal sesquilinearity conditions 
    \begin{align*}
& (\partial a)_{\lambda} m = - \lambda a_\lambda m,  \quad  a_{\lambda} (\partial m) = (\partial + \lambda)a_\lambda m,\\& (\partial m)_{\lambda} a = - \lambda m_\lambda a, \quad m_{\lambda} (\partial a) = (\partial + \lambda)m_\lambda a,
    \end{align*}
and the following bimodule conditions
    \begin{align*}
        a_{\lambda}(b_{\mu} m)  =  (a_{\lambda} b)_{\lambda+ \mu} m, \quad  a_{\lambda}(m_{\mu} b) =  (a_{\lambda} m)_{\lambda+ \mu} b ~~~
      \mathrm{ and } ~~~ m_{\lambda}(
      a_{\mu} b) =  (m_{\lambda} a)_{\lambda+ \mu} b,  \text{ for all } a, b \in A, m \in M.
    \end{align*}
\end{definition}

\medskip

It follows from the above definition that any associative conformal algebra $A$ is itself a conformal $A$-bimodule, where both the left and right $\lambda$-actions are given by the $\lambda$-multiplication. This is called the adjoint conformal $A$-bimodule.

In the following, we recall the Hochschild cohomology of associative conformal algebras introduced in \cite{bakalov-kac-voronov}. This cohomology is the conformal analogue of the standard Hochschild cohomology theory of associative algebras. Let $A$ be an associative conformal algebra and $M$ be a conformal $A$-bimodule. Let $n \geq 1$ be a natural number. A $\mathbb{C}$-linear map $\varphi : A^{\otimes n} \to M[\lambda_{1},\ldots,\lambda_{n-1}]$, $a_{1}\otimes \cdots\otimes a_{n} \mapsto \varphi_{\lambda_{1},\ldots,\lambda_{n-1}}(a_{1},\ldots,a_{n}) $ is called a {\em conformal sesquilinear map} if $\varphi$ satisfies
\begin{align*}
  \varphi_{\lambda_{1},\ldots,\lambda_{n-1}}(a_{1}, \ldots,\partial a_{i},\ldots, a_{n}) =~& -\lambda_{i}\varphi_{\lambda_{1},\ldots,\lambda_{n-1}}(a_{1},\ldots,a_{n}), \mathrm{ ~ for ~} i = 1, 2, \ldots ,n-1, \\
 \varphi_{\lambda_{1},\ldots,\lambda_{n-1}}(a_{1},\ldots,a_{n-1},\partial a_{n}) =~& (\partial + \lambda_{1} + \cdots + \lambda_{n-1})\varphi_{\lambda_{1},\ldots,\lambda_{n-1}}(a_{1},\ldots,a_{n}).
\end{align*}
We denote the set of all such conformal sesquilinear maps by $\mathrm{Hom}_{cs}(A^{\otimes n}, M[\lambda_{1},\ldots,\lambda_{n-1}])$. Note that, for $n = 1$, we have 
\begin{align*}
    \mathrm{Hom}_{cs}(A,M) = \mathrm{Hom}_{\mathbb{C}[\partial]}(A,M),
\end{align*}
%where $\mathrm{Hom}_{\mathbb{C}[\partial]}(A,M)$ is 
the set of all $\mathbb{C}[\partial]$-module homomorphisms from $A$ to $M$. We are now ready to define the Hochschild cochain complex of the associative conformal algebra $A$ with coefficients in the conformal $A$-bimodule $M$.\\
For each $n\geq 0$, the space of $n$-cochains $C^{n}(A,M)$ is given by  
\begin{equation*}
C^{n}(A,M)=
    \begin{cases}
        M/{\partial M} & \text{if  } n = 0,\\
        \mathrm{Hom}_{cs}(A^{\otimes n}, M[\lambda_{1},\ldots,\lambda_{n-1}]) & \text{if  }  n \geq 1.
    \end{cases}
\end{equation*}   
There is a map $\delta :C^{n}(A,M) \to  C^{n+1}(A,M)$ given by 
\begin{align*}
    \delta( m + \partial{M})(a) = (a_{-\lambda - \partial}m - m_{\lambda} a) ~ \! \vline_{\lambda = 0}, \mathrm{~for~}m + \partial{M} \in M/{\partial{M}}\mathrm{~and~}a \in A,
    \end{align*}
 and 
\begin{align*}
(\delta{\varphi})_{\lambda_{1},\ldots,\lambda_{n}}(a_{1},\ldots,a_{n+1}) =&~{a_{1}}_{\lambda_{1}} ( \varphi_{\lambda_{2},\ldots,\lambda{n}}(a_{2},\ldots,a_{n+1})) \\
+& \sum_{i=1}^{n}{(-1)^{i} ~ \varphi_{\lambda_{1},\ldots,\lambda_{i-1},\lambda_{i}+\lambda_{i+1},\lambda_{i+2},\ldots,\lambda_{n}}(a_{1},\ldots,a_{i-1},{a_{i}}_{\lambda_{i}}a_{i+1},a_{i+2},\ldots,a_{n+1})}\\
+&(-1)^{n+1} ( {\varphi_{\lambda_{1},\ldots,\lambda_{n-1}}(a_{1},\ldots,a_{n})})_{{\lambda_{1}+\cdots+\lambda_{n}}} a_{n+1},
\end{align*}
for $\varphi \in {C^{n \geq 1}(A,M)} = \mathrm{Hom}_{cs}(A^{\otimes n}, M[\lambda_{1},\ldots,\lambda_{n-1}]) $ and $a_{1},\ldots,a_{n+1} \in A$.
One can easily verify that $\delta$ is a differential, that is, $\delta^{2} = 0$. In other words, $\{C^{\bullet}(A, M), \delta \}$ is a cochain complex. Let $Z^{n}(A,M) = \{ \varphi \in C^{n}(A,M ) ~ |~\delta{\varphi } = 0\}$ be the space of all $n$-cocycles and $B^{n}(A,M) = \{ \delta{\varphi } ~ | ~ \varphi \in C^{n-1}(A,M )  \}$ be the space of all $n$-coboundaries. The quotients 
\begin{align*}
  {H}^{n}(A,M) := {\frac{{Z}^{n}(A,M)}{{{B}^{n}(A,M)}}, ~\mathrm{ for }~ n \geq 0} 
\end{align*}
are called the {\em Hochschild cohomology groups} of the associative conformal  algebra $A$ with coefficients in the conformal $A$-bimodule $M$.

\section{Homotopy associative conformal algebras} \label{section-3}
In this section, we introduce $A_{\infty}$-conformal algebras as the homotopy analogue of associative conformal algebras. Given a graded $\mathbb{C}[\partial]$-module $\mathcal{A}$, we also construct a graded Lie algebra whose Maurer-Cartan elements correspond to $A_\infty$-conformal algebra structures on $\mathcal{A}$. Next, we define representations of an $A_\infty$-conformal algebra and introduce the cohomology with coefficients in a representation. Finally, we prove homotopy transfer theorems for $A_\infty$-conformal algebras.

%We also define conformal bimodules over conformal $A_{\infty}$-algebras and introduce the cohomology theory of conformal $A_{\infty}$-algebras.
 
%We start with the definition of an $A_{\infty}$-algebra (also called a strongly homotopy associative algebra). Then we will generalize it to the conformal contexts.
\begin{definition}
    (\cite{stasheff}) An {\em $A_{\infty}$-algebra} is a pair $(\mathcal{A}, \{\mu_{k}\}_{k \geq 1})$ consisting of a graded vector space $\mathcal{A} = \oplus _{ i \in \mathbb{Z}} \mathcal{A}_{i}$ equipped with a collection of graded linear maps $\{\mu_{k} : \mathcal{A}^{\otimes k } \to \mathcal{A}\}_{k \geq 1}$ with $\mathrm{deg}(\mu _{k}) = k-2 $ for $k \geq 1$, satisfying the following set of identities:
    \begin{align*}
        \sum_{k+l = n+1}\sum_{i =1}^{k}{(-1)^{i(l+1) +l(|a_{1}|+\cdots+|a_{i-1}|)}~ \mu_{k}(a_{1},\ldots,a_{i-1},\mu_{l}(a_{i},\ldots,a_{i+l-1}),a_{i+1},\ldots,a_{n})} = 0,
    \end{align*}
    for all $n \geq 1$ and homogeneous elements $a_{1},\ldots,a_{n} \in \mathcal{A}$.
\end{definition}
Note that, in an arbitrary $A_{\infty}$-algebra, the degree $0$ multiplication map $\mu_2 : \mathcal{A} \otimes \mathcal{A} \rightarrow \mathcal{A}$ in general doesn't satisfy the usual associativity. However, it does satisfy the associativity up to a coherent homotopy. For this reason, $A_\infty$-algebras are also called {\em strongly homotopy associative algebras}. Any associative algebra can be regarded as an $A_{\infty}$-algebra whose underlying graded vector space is concentrated in degree $0$.

Let $\mathcal{A} = \oplus _{i \in \mathbb{Z}} \mathcal{A}_{i}$ be a graded vector space. Then for any $k \geq 1$, the space $\mathcal{A}[\lambda_{1},\ldots,\lambda_{k-1}]$ can be equipped with a graded vector space structure by 
\begin{align*}
    \mathcal{A}[\lambda_{1},\ldots,\lambda_{k-1}] = \oplus_{i \in \mathbb{Z}} \mathcal{A}_{i}[\lambda_{1},\ldots,\lambda_{k-1}].
\end{align*}

\begin{definition}\label{definition 3.2}
    An {\em $A_{\infty}$-conformal algebra} (or a {\em strongly homotopy associative conformal algebra}) is a graded $\mathbb{C}[\partial]$-module $\mathcal{A} = \oplus _{i \in \mathbb{Z}} \mathcal{A}_{i}$ equipped with a collection $\{\mu_{k}: \mathcal{A}^{\otimes k } \to \mathcal{A}[\lambda_{1},\ldots,\lambda_{k-1}]\}_{k \geq 1}$ of graded $\mathbb{C}$-linear maps with $\mathrm{deg}(\mu_{k}) = k-2$ for $k \geq 1$, satisfying the following conditions:
    
   - each $\mu_k$ is conformal sesquilinear in the sense that
    \begin{align*}
      (\mu_{k})_{\lambda_{1},\ldots,\lambda_{k-1}}(a_{1},\ldots,\partial{a_{i}},\ldots,a_{k}) =~& - \lambda_{i}(\mu_{k})_{\lambda_{1},\ldots,\lambda_{k-1}}(a_{1},\ldots,a_{k}), \mathrm{~for~}  i = 1, \ldots ,k-1,\\
        (\mu_{k})_{\lambda_{1},\ldots,\lambda_{k-1}} (a_{1},\ldots,a_{k-1},\partial{a_k}) =~& (\partial + \lambda_{1}+\cdots+ \lambda_{k-1})(\mu_{k})_{\lambda_{1},\ldots,\lambda_{k-1}}(a_{1},\ldots,a_{k}),
    \end{align*}
   
    - the following set of identities are hold:
    \begin{align} \label{a-inf-conf}
        \sum_{k+l=n+1} & \sum_{i=1}^{k}(-1)^{i(l+1) +l(|a_{1}|+\cdots+|a_{i-1}|)}\\&
     (\mu_{k})_{\lambda_{1},\ldots,\lambda_{i-1},\lambda_{i}+\cdots+\lambda_{i+l-1},\ldots,\lambda_{n-1}}(a_{1},\ldots,a_{i-1},(\mu_{l})_{\lambda_{i},\ldots,\lambda_{i+l-2}}(a_{i},\ldots,a_{i+l-1}),a_{i+l},\ldots,a_{n})= 0, \nonumber
    \end{align}
    for any $n \geq 1$ and homogeneous elements $a_{1},\ldots,a_{n} \in \mathcal{A}$.
\end{definition}

The collection $\{\mu_{k} : \mathcal{A}^{\otimes k } \to \mathcal{A}[\lambda_{1},\ldots,\lambda_{k-1}]\}_{k \geq 1}$ of maps are called the structure maps of the $A_{\infty}$-conformal algebra. We often denote an $A_{\infty}$-conformal algebra as above by $\mathcal{A}$ when the structure maps are clear from the context.

The above definition of an $A_\infty$-conformal algebra has the following consequences. Note that, for $n=1$, the identity (\ref{a-inf-conf}) is equivalent to $(\mu_1)^2 = 0$. In other words, the $\mathbb{C} [\partial]$-linear degree $-1$ map $\mu_1 : \mathcal{A} \rightarrow \mathcal{A}$ makes $(\mathcal{A}, \mu_1)$ into a chain complex in the category of $\mathbb{C} [\partial]$-modules. Next, for $n=2$, the identity (\ref{a-inf-conf}) is same as
\begin{align*}
    \mu_1 \big(    (\mu_2)_\lambda (a, b) \big) = (\mu_2)_\lambda \big(  \mu_1 (a) , b  \big) + (-1)^{|a|} (\mu_2)_\lambda \big(  a, \mu_1 (b) \big),
\end{align*}
for all homogeneous elements $a, b \in \mathcal{A}$. It says that $\mu_1$ is a graded derivation for the degree zero $\lambda$-multiplication $\mu_2 : \mathcal{A} \otimes \mathcal{A} \rightarrow \mathcal{A} [\lambda]$. Similarly, for $n=3$, we have
\begin{align*}
    &(\mu_2)_\lambda \big(  a, (\mu_2)_\mu (b, c) \big) - (\mu_2)_{\lambda + \mu} \big( (\mu_2)_\lambda (a, b), c   \big) \\
   & = - \big\{  \mu_1 \big(  (\mu_3)_{\lambda, \mu} (a, b, c)  \big) + (\mu_3)_{\lambda , \mu} \big(  \mu_1 (a), b, c \big) + (-1)^{|a|} (\mu_3)_{\lambda , \mu} (a, \mu_1 (b) , c) \\
    & \qquad \qquad + (-1)^{|a| + |b|} (\mu_3)_{\lambda, \mu} (a, b, \mu_1 (c))   \big\},
\end{align*}
for all homogeneous elements $a, b, c \in \mathcal{A}$. This shows that the $\lambda$-multiplication $\mu_2$ doesn't satisfy (in general) the conformal associativity. However, it holds up to an exact term of the conformal sesquilinear map $\mu_3$. For this reason, an $A_\infty$-conformal algebra can be realized as the homotopy analogue of associative conformal algebras. Note that, for higher values of $n$, we have higher homotopy analogues of the conformal associativity.

Let $\mathcal{A}$ be an $A_\infty$-conformal algebra with the structure maps $\{ \mu_k : \mathcal{A}^{\otimes k} \rightarrow \mathcal{A}[\lambda_1, \ldots, \lambda_{k-1}] \}_{k \geq 1}$. For any $k \geq 1$ and a $(k-1)$ tuple $J= (j_1, \ldots, j_{k-1})$ of non-negative integers, we define a graded $\mathbb{C}$-linear map $(\mu_k)_{(j_1, \ldots, j_{k-1})} : \mathcal{A}^{\otimes k} \rightarrow \mathcal{A}$ with  $\mathrm{deg } \big( 
 (\mu_k)_{(j_1, \ldots, j_{k-1})} \big) = k-2$, by
 \begin{align*}
     (\mu_k)_{\lambda_1, \ldots, \lambda_{k-1}} (a_1, \ldots, a_k ) = \sum_{(j_1, \ldots, j_{k-1}) \geq 0} \frac{\lambda_1^{j_1} \cdots \lambda_{k-1}^{j_{k-1}}}{j_1 ! \cdots j_{k-1} !} (\mu_k)_{(j_1, \ldots, j_{k-1})} (a_1, \ldots, a_k),
 \end{align*}
 for any homogeneous elements $a_1, \ldots, a_k \in \mathcal{A}$. We call them $J$-products, which generalize the $j$-products (\ref{lambda-j}) from the non-homotopic case. Since $ (\mu_k)_{\lambda_1, \ldots, \lambda_{k-1}} (a_1, \ldots, a_k )$ lies in the polynomial space $\mathcal{A}[\lambda_1, \ldots, \lambda_{k-1}]$, it follows that $(\mu_k)_{(j_1, \ldots, j_{k-1})} (a_1, \ldots, a_k) \neq 0$ only for finitely many tuples $(j_1, \ldots, j_{k-1})$'s that depends on $\mu_k$ and the inputs $a_1, \ldots, a_k$. The conformal sesquilinearity of $\mu_k$ can be equivalently expressed in terms of the $J$-products as
\begin{align*}
    (\mu_k)_{(j_1, \ldots, j_{k-1})} (a_1, \ldots, \partial a_i, \ldots, a_k) =~& - j_1 (\mu_k)_{(j_1, \ldots, j_i -1, \ldots, j_{k-1})} (a_1, \ldots, a_k), \text{ for } 1 \leq i \leq k-1,\\
    (\mu_k)_{(j_1, \ldots, j_{k-1})} (a_1, \ldots, a_{k-1}, \partial a_k ) =~& \partial \big(    (\mu_k)_{(j_1, \ldots, j_{k-1})} (a_1, \ldots, a_k )  \big) \\
    &+ \sum_{i=1}^{k-1} j_i (\mu_k)_{(j_1, \ldots, j_i -1, \ldots, j_{k-1})} (a_1, \ldots, a_k),
\end{align*}
for any $(j_1, \ldots, j_{k-1}) \geq 0$ and $a_1, \ldots, a_k \in \mathcal{A}$. Finally, the $A_\infty$-conformal associativity conditions (\ref{a-inf-conf}) are equivalent to
\begin{align}
    \sum_{k+l = n+1} & \sum_{i=1}^k (-1)^{i (l+1) + l ( |a_1| + \cdots + |a_{i-1}|)} \sum_{\substack{ 0 \leq q_i \leq p_i \\ \vdots \\ 0 \leq q_{i+l-2} \leq p_{i+l-2}}} \binom{p_i}{q_i} \cdots \binom{p_{i+l-2}}{q_{i+l-2}} \\
   & (\mu_k)_{ (p_1, \ldots, p_{i-1}, p_i - q_i + \cdots + p_{i+l-2} - q_{i+l-2} + p_{i+l-1}, p_{i+l}, \ldots, p_{n-1})} \big(  a_1, \ldots, a_{i-1}, \nonumber \\ & \qquad \qquad \qquad \qquad  \qquad \qquad \qquad \qquad  
 (\mu_l)_{(q_i, \ldots, q_{i+l-2})} (a_i, \ldots, a_{i+l-1}), \ldots, a_{n}  \big) = 0, \nonumber
\end{align}
for any $p_1, \ldots, p_{n-1} \geq 0$, homogeneous elements $a_1, \ldots, a_n \in \mathcal{A}$ and $n \geq 1$. These relations generalize (\ref{conf-ass-j}).

\begin{remark}
Any associative conformal algebra $(A, \cdot_\lambda \cdot)$ can be viewed as an $A_{\infty}$-conformal algebra $(\mathcal{A},\{\mu_{k}\}_{k\geq 1})$, where $\mathcal{A} = A$ concentrated in degree $0$, 
\begin{align*}
(\mu_{2})_{\lambda} (a,b) = a_{\lambda}b ~~~ \text{ and } ~~~ \mu_{k} = 0, \mathrm{~for~}  k \neq 2.
\end{align*}
\end{remark}

\begin{exam}
    A {\em differential graded associative conformal algebra} is a graded $\mathbb{C}[\partial]$-module $\mathcal{A} = \oplus _{i \in \mathbb{Z}} \mathcal{A}_{i}$ equipped with a $\mathbb{C}[\partial]$-module homomorphism $ d : \mathcal{A} \rightarrow \mathcal{A}$ of degree $-1$ that satisfies $d^2 = 0$ and a degree $0$ conformal sesquilinear map
$\cdot_\lambda \cdot: {\mathcal{A}\otimes \mathcal{A} } \to {\mathcal{A}[\lambda]}$,  $a\otimes b \mapsto {a_{\lambda} b}$ that satisfies  
\begin{align*}
   a_{\lambda}(b_{\mu} c) = (a_{\lambda}b)_{\lambda+\mu }c ~~~ \text{ and } ~~~
   d(a_{\lambda}b) = (da)_{\lambda}b + a_{\lambda}(db), \mathrm{~for~ all }~ a,b,c \in \mathcal{A}.
\end{align*}
A differential graded associative  conformal algebra $(\mathcal{A}, d, \cdot_\lambda \cdot)$ is an $A_{\infty}$-conformal algebra $(\mathcal{A},\{\mu_{k}\}_{k \geq 1})$ in which $\mu_{1} = d$, $(\mu_{2})_\lambda = \cdot_\lambda \cdot$ and $\mu_{k} = 0$ for $k \geq 3$.
\end{exam}

\begin{exam}
Let $A$ be an associative conformal algebra. Then the graded $\mathbb{C}[\partial]$-module $\mathcal{A} = \underbrace{A}_{\mathrm{deg } 0} \oplus \underbrace{A}_{ \mathrm{deg } 1}$ carries an $A_\infty$-conformal algebra structure with the $\lambda$-multiplications $\{ \mu_k : \mathcal{A}^{\otimes k} \rightarrow \mathcal{A} [\lambda_1, \ldots, \lambda_{k-1}] \}_{k \geq 1}$ given by
\begin{align*}
    \mu_1 = \mathrm{id} : \mathcal{A}_1 \rightarrow \mathcal{A}_0, \quad \mu_2 : \mathcal{A}_i \otimes \mathcal{A}_j \rightarrow \mathcal{A}_{i+j}[\lambda],~ a \otimes b \mapsto a_\lambda b ~\text{ for } 0 \leq i, j, i+j \leq 1, \text{ and } \mu_k = 0 \text{ for } k \geq 3.
\end{align*}
\end{exam}

The above example can be generalized as follows. Let $A$ and $B$ be two associative conformal algebras and $f : A \rightarrow B$ be a morphism of associative conformal algebras (i.e. $f$ is a $\mathbb{C}[\partial]$-linear map satisfying $f (a_\lambda b) = f(a)_\lambda f(b)$, for $a, b \in A$). Then the graded $\mathbb{C}[\partial]$-module $\mathcal{A} = \underbrace{A}_{\mathrm{deg } 0} \oplus \underbrace{\mathrm{ker }f}_{\mathrm{deg }1}$ inherits an $A_\infty$-conformal algebra structure with the $\lambda$-multiplications $\{ \mu_k : \mathcal{A}^{\otimes k} \rightarrow \mathcal{A} [\lambda_1, \ldots, \lambda_{k-1}] \}_{k \geq 1}$ given by
\begin{align*}
    \mu_1 = i : \mathcal{A}_1 \hookrightarrow \mathcal{A}_0, \quad \mu_2 : \mathcal{A}_i \otimes \mathcal{A}_j \rightarrow \mathcal{A}_{i+j} [\lambda], ~ a \otimes b \mapsto a_\lambda b \text{ for } 0 \leq i, j, i+j \leq 1, ~\text{ and } \mu_k = 0 \text{ for } k \geq 3.
\end{align*}

\begin{exam}
    Let $A$ be an associative conformal algebra, $M$ be a conformal $A$-module and let $\varphi: A \otimes A \rightarrow M[\lambda]$ be any conformal sesquilinear map (i.e. $\varphi \in C^2 (A, M)$). Then the graded $\mathbb{C}[\partial]$-module $\mathcal{A} = \underbrace{A}_{\mathrm{deg } 0} \oplus \underbrace{M}_{\mathrm{deg } 1}$ has an $A_\infty$-conformal algebra with the maps $\{ \mu_k : \mathcal{A}^{\otimes k} \rightarrow \mathcal{A}[\lambda_1, \ldots, \lambda_{k-1}] \}_{k \geq 1}$ given by
    \begin{align*}
        &\mu_1 = 0 : \mathcal{A}_1 \rightarrow \mathcal{A}_0, \qquad \begin{cases}
            (\mu_2)_\lambda (a, b) := a_\lambda b,\\
            (\mu_2)_\lambda (a, m) := a_\lambda m,\\
            (\mu_2)_\lambda (m,a) := m_\lambda a,
        \end{cases}\\
       & (\mu_3)_{\lambda, \mu} (a, b, c) :=  a_\lambda \varphi_\mu (b, c) - \varphi_{\lambda + \mu} (a_\lambda b, c) + \varphi_\lambda (a, b_\mu c) - \varphi_\lambda (a, b)_{\lambda + \mu} c, 
    \end{align*}
    for $a, b, c \in A, m \in M$, and $\mu_k = 0$ for $k \geq 4$.
\end{exam}

The concept of $A_{\infty}$-conformal algebras are better understood when we shift the degree of the underlying graded $\mathbb{C}[\partial]$-module. Let $\mathcal{A} = \oplus_{i \in \mathbb{Z}}\mathcal{A}_{i}$ be a graded $\mathbb{C}[\partial]$-module. We consider the graded $\mathbb{C}[\partial]$-module $\mathcal{A}[-1] = \oplus_{i \in \mathbb{Z}}(\mathcal{A}[-1])_{i} $, where $\mathcal{A}[-1]_{i} = \mathcal{A}_{i-1}$. 
%Then $\mathcal{A}[-1]$ is also a graded $\mathbb{C}[\partial]$-module with the obvious action of $\partial$. 
Note that a graded $\mathbb{C}$-linear map  $\mu_{k} : \mathcal{A}^{\otimes k } \to \mathcal{A}[\lambda_{1},\ldots,\lambda_{k-1}]$ with deg$(\mu_{k}) = k-2$ is equivalent to having a graded $\mathbb{C}$-linear map $\rho_{k} : (\mathcal{A}[-1])^{\otimes k} \to (\mathcal{A}[-1])[\lambda_{1},\ldots,\lambda_{k-1}]$ with deg$(\rho_{k}) = -1$. The correspondence is given by 
\begin{align*}
    \rho_{k} := (-1)^{\frac{k(k-1)}{2}}~s\circ \mu_{k} \circ (s^{-1})^{\otimes k},
\end{align*}
where $s: \mathcal{A} \to \mathcal{A}[-1] $ is the degree $+1$ map that identifies $\mathcal{A}$ with $\mathcal{A}[-1]$, and $s^{-1}: \mathcal{A}[-1] \to \mathcal{A}$
is the degree $-1$ map inverse to $s$. The sign $(-1)^{\frac{k(k-1)}{2}}$ is a consequence of the Koszul sign convention. See \cite{lada-markl} for more details about the Koszul sign. We are now ready to define the concept of $A_{\infty}[1]$-conformal algebras and their relation with Definition \ref{definition 3.2}.

\begin{definition}
    An {\em $A_{\infty}[1]$-conformal algebra} is a graded $\mathbb{C}[\partial]$-module $\mathcal{V} = \oplus_{i \in \mathbb{Z}} \mathcal{V}_{i}$ together with a collection $\{ \rho_{k} : \mathcal{V}^{\otimes k} \to \mathcal{V}[\lambda_{1},\ldots,\lambda_{k-1}]\}_{k \geq 1}$ of degree $-1$ graded $\mathbb{C}$-linear maps that satisfy the following conditions:
    
   - each $\rho_{k}$ is conformal sesquilinear,
   
   - for any $n \geq 1 $ and homogenous elements $v_{1},\ldots,v_{n} \in \mathcal{V}$,
   \begin{align}\label{a1-inf-conf}
     \sum_{k+l=n+1} & \sum_{i=1}^{k}  (-1)^{|v_{1}|+\cdots+|v_{i-1}|}\\&
     (\rho_{k})_{\lambda_{1},\ldots,\lambda_{i-1},\lambda_{i}+\cdots+\lambda_{i+l-1},\ldots,\lambda_{n-1}}(v_{1},\ldots,v_{i-1},(\rho_{l})_{\lambda_{i},\ldots,\lambda_{i+l-2}}(v_{i},\ldots,v_{i+l-1}),v_{i+l},\ldots,v_{n})= 0. \nonumber
   \end{align}
\end{definition}

\medskip

It is easy to see that $A_\infty[1]$-conformal algebras are easier to keep in mind than $A_\infty$-conformal algebras because of the fixed degree of the structure maps and the signs in the defining identities. However, they are equivalent up to a degree shift.

\begin{proposition}\label{proposition 3.6}
An $A_{\infty}$-conformal algebra structure on a graded $\mathbb{C}[\partial]$-module $\mathcal{A}$ is equivalent to an $A_{\infty}[1]$-conformal algebra structure on the graded $\mathbb{C}[\partial]$-module $\mathcal{A}[-1]$. More precisely, $(\mathcal{A},\{\mu_{k}\}_{k \geq 1})$ is an $A_{\infty}$-conformal algebra if and only if $(\mathcal{A}[-1], \{\rho_{k}\}_{k \geq 1})$ is an $A_{\infty}[1]$-conformal algebra.
\end{proposition}

%\begin{remark}
   
%\end{remark}

\medskip

 % \medskip
  
\noindent {\bf Maurer-Cartan characterization and cohomology of $A_\infty$-conformal algebras.} In the following, we show that $A_{\infty}$-conformal algebra structures on a graded $\mathbb{C}[\partial]$-module can be equivalently described by Maurer-Cartan elements in a suitable graded Lie algebra. Let $\mathcal{V} = \oplus_{i \in \mathbb{Z}} \mathcal{V}_{i}$ be a graded $\mathbb{C}[\partial]$-module. For each $k\geq 1$, let $\mathrm{Hom}^{n}_{cs}(\mathcal{V}^{\otimes k}, \mathcal{V}[\lambda_{1},\ldots, \lambda_{k-1}])$ be the set of all degree $n$ conformal sesquilinear maps from $\mathcal{V}^{\otimes k}$ to $\mathcal{V}[\lambda_{1},\ldots, \lambda_{k-1}]$. That is,
\begin{align*}
   \mathrm{Hom}^{n}_{cs}(\mathcal{V}^{\otimes k}, \mathcal{V}[\lambda_{1},\ldots, \lambda_{k-1}]) = \{\varphi \in \mathrm{Hom}_{cs}(\mathcal{V}^{\otimes k}, \mathcal{V}[\lambda_{1},\ldots, \lambda_{k-1}])~ | \mathrm{~deg}(\varphi) = n\}.
\end{align*}
We define a graded vector space $C^{\bullet}_{cs}(\mathcal{V}) = \oplus_{n \in \mathbb{Z}} C^{n}_{cs}(\mathcal{V})$, where 
\begin{align*}
    C^{n}_{cs}(\mathcal{V})= \oplus_{k \geq1}\mathrm{Hom}^{n}_{cs}(\mathcal{V}^{\otimes k}, \mathcal{V}[\lambda_{1},\ldots, \lambda_{k-1}]).
\end{align*}
Thus, an element $\rho \in C^{n}_{cs}(\mathcal{V}) $ is given by a sum $\rho = \sum_{k\geq1}^{}\rho_{k}$, where $\rho_{k}\in \mathrm{Hom}^{n}_{cs}(\mathcal{V}^{\otimes k}, \mathcal{V}[\lambda_{1},\ldots, \lambda_{k-1}])$ for any $k \geq 1$. Note that the graded vector space $C^{\bullet}_{cs}(\mathcal{V}) = \oplus_{n \in \mathbb{Z}} C^{n}_{cs}(\mathcal{V})$ inherits a degree $0$ bracket given by
%graded Lie algebra structure with the bracket
\begin{align*}
    \llbracket \sum_{k\geq1}^{}\rho_{k},\sum_{l\geq1}^{}\varrho_{l} \rrbracket = \sum_{p \geq 1}\sum_{k+l = p+1}(\rho_{k}\diamond \varrho_{l}-(-1)^{mn}\varrho_{l} \diamond \rho_{k}),
\end{align*}
where
\begin{align}\label{diam-product}
    (\rho_{k}\diamond \varrho_{l}&)_{\lambda_{1},\ldots,\lambda_{p-1}}(v_{1},\ldots,v_{p}) =\sum_{i=1}^{k}(-1)^{m(|v_{1}|+\cdots+|v_{i-1}|)}\\&
    (\rho_{k})_{\lambda_{1},\ldots,\lambda_{i-1},\lambda_{i}+\cdots+\lambda_{i+l-1},\ldots,\lambda_{p-1}}(v_{1},\ldots,v_{i-1},(\varrho_{l})_{\lambda_{i},\ldots,\lambda_{i+l-2}}(v_{i},\ldots,v_{i+l-1}),v_{i+1},\ldots,v_{p}), \nonumber
\end{align}
    for $\rho = \sum_{k \geq 1}\rho_{k}\in C^{n}_{cs}(\mathcal{V}) $ and $\varrho = \sum_{l \geq 1} \varrho_{l}\in C^{m}_{cs}(\mathcal{V}) $. Then we have the following.
    
    %In other words, the pair $(C^{\bullet}_{cs}(\mathcal{V}), \llbracket ~,~ \rrbracket )$ is a graded Lie algebra. With this notation, we have the following result.

    \begin{proposition}
        With the above notations, $(C^\bullet_{cs} (\mathcal{V}), \llbracket ~, ~ \rrbracket)$ is a graded Lie algebra.
    \end{proposition}

    \begin{proof}
        For any $\rho_k \in \mathrm{Hom}^n_{cs} (\mathcal{V}^{\otimes k} , \mathcal{V} [\lambda_1, \ldots, \lambda_{k-1}])$, $\varrho_l \in \mathrm{Hom}^m_{cs} (\mathcal{V}^{\otimes l} , \mathcal{V} [\lambda_1, \ldots, \lambda_{l-1}])$ and $1 \leq i \leq k$, we define an element $\rho_k \diamond_i \varrho_l \in \mathrm{Hom}^{n+m}_{cs} (\mathcal{V}^{\otimes (k+l-1)} , \mathcal{V} [\lambda_1, \ldots, \lambda_{k+l-2}])$ by
        \begin{align*}
    (\rho_{k}\diamond_i \varrho_{l}&)_{\lambda_{1},\ldots,\lambda_{k+l-2}}(v_{1},\ldots,v_{k+l-1}) = (-1)^{m(|v_{1}|+\cdots+|v_{i-1}|)}\\&
    (\rho_{k})_{\lambda_{1},\ldots,\lambda_{i-1},\lambda_{i}+\cdots+\lambda_{i+l-1},\ldots,\lambda_{k+l-2}}(v_{1},\ldots,v_{i-1},(\varrho_{l})_{\lambda_{i},\ldots,\lambda_{i+l-2}} (v_{i},\ldots,v_{i+l-1}), v_{i+1},\ldots,v_{k+l-1}), \nonumber
\end{align*}
Then we have
\begin{align}\label{sk-br}
    \llbracket \rho_k , \varrho_l \rrbracket = \rho_k \diamond \varrho_l - (-1)^{mn} \varrho_l \diamond \rho_k = \sum_{i=1}^k \rho_k \diamond_i \varrho_l - (-1)^{mn} \sum_{i=1}^l \varrho_l \diamond_i \rho_k.
\end{align}
Moreover, for any elements $\rho_k \in \mathrm{Hom}^n_{cs} (\mathcal{V}^{\otimes k} , \mathcal{V} [\lambda_1, \ldots, \lambda_{k-1}])$, $\varrho_l \in \mathrm{Hom}^m_{cs} (\mathcal{V}^{\otimes l} , \mathcal{V} [\lambda_1, \ldots, \lambda_{l-1}])$ and $\psi_s \in \mathrm{Hom}^p_{cs} (\mathcal{V}^{\otimes s}, \mathcal{V} [\lambda_1, \ldots, \lambda_{s-1}])$, it is easy to verify that the $\diamond_i$ operations satisfy the following operad-like identities (see also \cite{hou} for the non-graded case)
\begin{align*}
    (\rho_k \diamond_i \varrho_l) \diamond_{i+j-1} \psi_s =~& \rho_k \diamond_i (\varrho_l \diamond_j \psi_s), \text{ for } 1 \leq i \leq k, ~ 1 \leq j \leq l, \\
    (\rho_k \diamond_i \varrho_l) \diamond_{j+l-1} \psi_s =~& (\rho_k \diamond_j \varrho_l) \diamond_i \psi_s, \text{ for } 1 \leq i < j \leq k.
\end{align*}
As a consequence, the graded skew-symmetric bracket (\ref{sk-br}) satisfies the graded Jacobi identity:
\begin{align*}
    \llbracket \rho_k, \llbracket \varrho_l, \psi_s \rrbracket \rrbracket = \llbracket \llbracket \rho_k, \varrho_l \rrbracket, \psi_s \rrbracket + (-1)^{mn} \llbracket \varrho_l, \llbracket \rho_k, \psi_s \rrbracket \rrbracket.
\end{align*}
Finally, if $\rho = \sum_{k \geq 1} \rho_k \in C^n_{cs} (\mathcal{V})$, $\varrho = \sum_{l \geq 1} \varrho_l \in C^m_{cs} (\mathcal{V})$ and $\psi = \sum_{s \geq 1} \psi_s \in C^p_{cs} (\mathcal{V})$, then we have
\begin{align*}
    \llbracket \rho, \llbracket \varrho, \psi \rrbracket \rrbracket =~& \sum_{r \geq 1} \sum_{k+l+s = r+2} \llbracket \rho_k, \llbracket \varrho_l, \psi_s \rrbracket \rrbracket \\
    =~&  \sum_{r \geq 1} \sum_{k+l+s = r+2} \big( \llbracket \llbracket \rho_k, \varrho_l \rrbracket, \psi_s \rrbracket + (-1)^{mn} \llbracket \varrho_l, \llbracket \rho_k, \psi_s \rrbracket \rrbracket  \big) \\
    =~& \llbracket \llbracket \rho, \varrho \rrbracket, \psi \rrbracket + (-1)^{mn} \llbracket \varrho, \llbracket \rho, \psi \rrbracket \rrbracket,
\end{align*}
which verifies the graded Jacobi identity for general elements.
    \end{proof}
 
    \begin{thm}\label{Theorem 1}
    Let $\mathcal{V} = \oplus_{i \in \mathbb{Z}} \mathcal{V}_{i}$ be a graded $\mathbb{C}[\partial]$-module. A collection of degree $-1$ conformal sesquilinear maps $\{ \rho_{k} : \mathcal{V}^{\otimes k} \to \mathcal{V}[\lambda_{1},\ldots,\lambda_{k-1}]\}_{k \geq 1}$ makes the pair $(\mathcal{V}, \{\rho_{k}\}_{k \geq 1})$ into an $A_{\infty}[1]$-conformal algebra if and only if the element $\rho = \sum_{k \geq 1}\rho_{k} \in C_{cs}^{-1}(\mathcal{V})$ is a Maurer-Cartan element in the graded Lie algebra $(C^{\bullet}_{cs}(\mathcal{V}), \llbracket ~,~ \rrbracket)$.
\end{thm}

\begin{proof}
We observe that
\begin{align*}
    \llbracket \rho, \rho \rrbracket = \llbracket \sum_{k \geq 1} \rho_k , \sum_{l \geq 1} \rho_l \rrbracket =~& \sum_{n \geq 1} \sum_{k+l = n+1} (\rho_k \diamond \rho_l - (-)^1 \rho_l \diamond \rho_k ) \\
    =~& \sum_{n \geq 1} \big(  2 \sum_{k+l = n+1} \rho_k \diamond \rho_l  \big).
\end{align*}
This shows that $\rho = \sum_{k \geq 1} \rho_k$ is a Maurer-Cartan element in the graded Lie algebra $(C^\bullet_{cs} (\mathcal{V}), \llbracket ~, ~ \rrbracket)$ if and only if $\sum_{k+l = n+1} \rho_k \diamond \rho_l = 0$ for all $n \geq 1$. Hence the result follows by using (\ref{diam-product}).
\end{proof}

 Thus, by combining Proposition \ref{proposition 3.6} and Theorem \ref{Theorem 1}, we conclude that the following statements are equivalent:
 \begin{enumerate}
     \item $(\mathcal{A}, \{ \mu_{k}\}_{k \geq 1})$ is an $A_{\infty}$-conformal algebra,
     \item $(\mathcal{A}[-1],\{\rho_{k}\}_{k \geq 1})$ is an $A_{\infty}[1]$-conformal algebra,
     \item $\rho = \sum_{k \geq 1} \rho_{k}$ is a Maurer-Cartan element in the graded Lie algebra $(C^{\bullet}_{cs}(\mathcal{A}[-1]),[\![~,~]\!])$.
 \end{enumerate}
% i.e. (1), (2) and (3) are equivalent.

\medskip

Let $(\mathcal{A}, \{ \mu_k \}_{k \geq 1})$ be an $A_\infty$-conformal algebra. Consider the corresponding Maurer-Cartan element $\rho = \sum_{k \geq 1} \rho_k$ in the graded Lie algebra $(C^\bullet_{cs} (\mathcal{A}[-1]), \llbracket ~, ~ \rrbracket)$. For each $n \in \mathbb{Z}$, we define the space of $n$-cochains by $C^n (\mathcal{A}, \mathcal{A}) := C^{- (n-1)}_{cs} (\mathcal{A}[-1])$. Thus, an element $\varphi \in C^n (\mathcal{A}, \mathcal{A})$ is given by a sum $\varphi = \sum_{l \geq 1} \varphi_l$, where $\varphi_l \in \mathrm{Hom}_{cs}^{- (n-1)} ( \mathcal{A}[-1]^{\otimes l} , \mathcal{A}[-1][\lambda_1, \ldots, \lambda_{l-1}])$ for any $l \geq 1$. We define a map $\delta_\rho : C^n (\mathcal{A}, \mathcal{A}) \rightarrow C^{n+1} (\mathcal{A}, \mathcal{A}) $ by
\begin{align*}
    \delta_\rho (\varphi ) := (-1)^{n-1} \llbracket \rho , \varphi \rrbracket = (-1)^{n-1} \sum_{p \geq 1} \sum_{k+l = p+1} \big( \rho_k \diamond \varphi_l - (-1)^{n-1} \varphi_l \diamond \rho_k  \big), \text{ for } \varphi = \sum_{l \geq 1} \varphi_l \in C^n (\mathcal{A}, \mathcal{A}). 
\end{align*}
Since $\rho$ is a Maurer-Cartan element (i.e. $\llbracket \rho, \rho \rrbracket = 0$), it follows that $(\delta_\rho)^2 = 0$. Hence $\{ C^\bullet (\mathcal{A}, \mathcal{A}), \delta_\rho \}$ is a cochain complex. The corresponding cohomology groups are called the {\em cohomology} of the $A_\infty$-conformal algebra $(\mathcal{A}, \{ \mu_k \}_{k \geq 1})$.

In the following, we introduce representations of an $A_{\infty}$-conformal algebra and define cohomology with coefficients in a representation. We first need the following notation. Let $\mathcal{A}$ and $\mathcal{M}$ be two graded $\mathbb{C}[\partial]$-modules. For any $k \geq 1$, let $\mathcal{A}^{k-1,1}$ be the direct sum of all possible tensor powers of $\mathcal{A}$ and $\mathcal{M}$ in which $\mathcal{A}$ appears $k-1$ times and $\mathcal{M}$ appears exactly once. For instance, 
\begin{align*}
    \mathcal{A}^{1,1} =(\mathcal{A} \otimes\mathcal{M})\oplus(\mathcal{M}\otimes\mathcal{A}) \quad  \mathrm{and} \quad
   \mathcal{A}^{2,1} = (\mathcal{A}\otimes \mathcal{A}\otimes\mathcal{M})\oplus(\mathcal{A}\otimes \mathcal{M}\otimes\mathcal{A})\oplus(\mathcal{M}\otimes \mathcal{A}\otimes\mathcal{A}).
\end{align*}

\begin{definition}
    Let $(\mathcal{A},\{\mu_{k}\}_{k \geq 1})$ be an $A_{\infty}$-conformal algebra. A {\em representation} of $(\mathcal{A},\{\mu_{k}\}_{k \geq 1})$ is a graded $\mathbb{C}[\partial]$-module $\mathcal{M} = \oplus_{i \in \mathbb{Z}} \mathcal{M}_{i}$ equipped with a collection $\{\eta_{k} : \mathcal{A}^{k-1,1} \to \mathcal{M}[\lambda_{1},\ldots,\lambda_{k-1}]\}_{k \geq 1}$ of graded conformal sesquilinear maps with $\mathrm{deg}(\eta_{k}) = k-2$ for any $k \geq 1$, such that the identities in (\ref{a-inf-conf}) are hold when exactly one of $a_{1},\ldots , a_{n}$ is from $\mathcal{M}$ and the corresponding $\mu_{k}$ or $\mu_{l}$ is replaced by $\eta_{k}$ or $\eta_{l}$.
\end{definition}

It follows from the above definition that any $A_{\infty}$-conformal algebra $(\mathcal{A},\{\mu_{k}\}_{k \geq 1})$ is a representation of itself, where $\eta_k = \mu_k$ for any $k \geq 1$. This is called the {\em adjoint representation}.

Let $(\mathcal{A}, \{ \mu_k \}_{k \geq 1})$ be an $A_\infty$-conformal algebra and $(\mathcal{M}, \{ \eta_k \}_{k \geq 1})$ be a representation of it. Then it can be checked that the direct sum graded $\mathbb{C}[\partial]$-module $\mathcal{A} \oplus \mathcal{M}$ inherits an $A_\infty$-conformal algebra with the structure maps $\{ \theta_k : (\mathcal{A} \oplus \mathcal{M})^{\otimes k} \rightarrow (\mathcal{A} \oplus \mathcal{M}) [\lambda_1, \ldots, \lambda_{k-1}]  \}_{k \geq 1}$ given by
\begin{align*}
    (\theta_k)_{\lambda_1, \ldots, \lambda_{k-1}} & \big(  (a_1,m_1), \ldots, (a_k, m_k) \big) :=\\ & \big(  (\mu_k)_{\lambda_1, \ldots, \lambda_{k-1}} (a_1, \ldots, a_k), ~ \sum_{i=1}^k (\eta_k)_{\lambda_1, \ldots, \lambda_{k-1}} (a_1, \ldots, a_{i-1}, m_i, a_{i+1}, \ldots, a_k) \big),
\end{align*}
for $(a_1, m_1), \ldots, (a_k, m_k) \in \mathcal{A} \oplus \mathcal{M}$. This is called the {\em semidirect product} $A_\infty$-conformal algebra, often denoted by $\mathcal{A} \ltimes \mathcal{M}$. Let $\Delta \in C^{-1}_{cs} ((\mathcal{A} \oplus \mathcal{M})[-1])$ be the Maurer-Cartan element in the graded Lie algebra $\big(  C^\bullet_{cs} ((\mathcal{A} \oplus \mathcal{M})[-1]) , \llbracket ~, ~ \rrbracket \big)$ corresponding to the semidirect product $A_\infty$-conformal algebra $(\mathcal{A} \oplus \mathcal{M}, \{ \theta_k \}_{k \geq 1})$. Explicitly, $\Delta$ is given by
\begin{align*}
    \Delta = \sum_{k \geq 1} \Delta_k, ~ \text{ where } \Delta_k = (-1)^{\frac{k(k-1)}{2}} ~ s \circ \theta_k \circ (s^{-1})^{\otimes k} \text{ for any } k \geq 1.
\end{align*}
The element $\Delta$ induces the cochain complex $\{ C^\bullet (\mathcal{A} \oplus \mathcal{M}, \mathcal{A} \oplus \mathcal{M}), \delta_\Delta \}$ defining the cohomology of the semidirect product $A_\infty$-conformal algebra. For any $n \in \mathbb{Z}$, let $C^n (\mathcal{A}, \mathcal{M}) \subset C^n (\mathcal{A} \oplus \mathcal{M}, \mathcal{A} \oplus \mathcal{M})$ be the subspace given by
\begin{align*}
    C^n (\mathcal{A}, \mathcal{M}) := \{ \varphi = \sum_{l \geq 1} \varphi_l \in C^n (\mathcal{A} \oplus \mathcal{M}, \mathcal{A} \oplus \mathcal{M}) ~|~ (\varphi_l)|_{\mathcal{A}[-1]^{\otimes l}} \subset \mathcal{M}[-1][\lambda_1, \ldots, \lambda_{l-1}] \text{ for any } l \geq 1 \}. 
\end{align*}
Note that, an element $\varphi \in C^n (\mathcal{A}, \mathcal{M})$ can be realized by a sum $\varphi = \sum_{l \geq 1} \varphi_l$, where 
\begin{align*}
\varphi_l \in \mathrm{Hom}_{cs}^{- (n-1)} ( \mathcal{A}[-1]^{\otimes l} , \mathcal{M}[-1][\lambda_1, \ldots, \lambda_{l-1}]), \text{ for any } l \geq 1.
\end{align*}
Then it follows from the explicit description of $\delta_\Delta$ that $\delta_\Delta \big(  C^n (\mathcal{A}, \mathcal{M}) \big) \subset C^{n+1} (\mathcal{A}, \mathcal{M})$ for all $n$. Hence $\delta_\Delta$ restricts to a differential (denoted by $\delta_\rho$ if no confusion arises) $\delta_\rho : C^n (\mathcal{A}, \mathcal{M}) \rightarrow  C^{n+1} (\mathcal{A}, \mathcal{M})$. In other words, $\{ C^\bullet (\mathcal{A}, \mathcal{M}), \delta_\rho \}$ is a cochain complex. The corresponding cohomology groups are called the {\em cohomology} of the $A_\infty$-conformal algebra $(\mathcal{A}, \{ \mu_k \}_{k \geq 1})$ with coefficients in the representation $(\mathcal{M}, \{ \eta_k \}_{k \geq 1})$.

\begin{remark}
    When $(\mathcal{M}, \{ \eta_k \}_{k \geq 1}) = (\mathcal{A}, \{ \mu_k \}_{k \geq 1})$ is the adjoint representation, the above cohomology coincides with the cohomology of the $A_\infty$-conformal algebra $(\mathcal{A}, \{ \mu_k \}_{k \geq 1})$ introduced earlier.
\end{remark}

\medskip

\noindent{\bf Homotopy transfer theorems.} Here we will prove the homotopy transfer theorems for $A_\infty$-conformal algebras. As we have observed earlier that $A_\infty[1]$-conformal algebras are easier to work, we will first prove the homotopy transfer theorems for $A_\infty [1]$-conformal algebras. Motivation for our results comes from the following equivalent description of $A_\infty [1]$-conformal algebras.

 Let $(\mathcal{V} = \oplus_{i \in \mathbb{Z}} \mathcal{V}^i, \rho_1)$ be a chain complex of $\mathbb{C}[\partial]$-modules. For any $k \geq 2$, we define a map $\partial_{\rho_1} : \mathrm{Hom}_{cs}^p (  \mathcal{V}^{\otimes k} , \mathcal{V}[\lambda_1, \ldots, \lambda_{k-1}]) \rightarrow  \mathrm{Hom}_{cs}^{p-1} (  \mathcal{V}^{\otimes k} , \mathcal{V}[\lambda_1, \ldots, \lambda_{k-1}])$ by
    \begin{align*}
        \partial_{\rho_1} (\rho_k) =  \llbracket \rho_1, \rho_k \rrbracket = \rho_1 \diamond \rho_k - (-1)^p \rho_k \diamond \rho_1 = \rho_1 \diamond_1 \rho_k  - (-1)^p   \sum_{i=1}^k \rho_k \diamond_i \rho_1,
    \end{align*}
%Explicitly, we have
%\begin{align*}
%     (  \partial (\rho_k) )_{\lambda_1, \ldots, \lambda_{k-1}} (v_1, \ldots, v_k ) = \textcolor{red}{left}
%\end{align*}
for $ \rho_k \in \mathrm{Hom}_{cs}^{p} (  \mathcal{V}^{\otimes k} , \mathcal{V}[\lambda_1, \ldots, \lambda_{k-1}]).$
With the above notations, the identities (\ref{a1-inf-conf}) for $n \geq 2$ can be equivalently written as
\begin{align}\label{par-a1}
    \partial_{\rho_1} (\rho_n) = \begin{cases}
        0 & \text{ if } n \geq 2,\\
        - \sum_{ \substack{k+l = n+1 \\ k, l >1}} \rho_k \diamond \rho_l =  - \sum_{  \substack{k+l = n+1 \\ k, l >1} } \sum_{i=1}^k \rho_k \diamond_i \rho_l & \text{ if } n \geq 3.
    \end{cases}
\end{align}
Thus, an $A_\infty [1]$-conformal algebra can be realized as a chain complex $(\mathcal{V} = \oplus_{i \in \mathbb{Z}} \mathcal{V}^i, \rho_1)$ of $\mathbb{C}[\partial]$-modules equipped with a collection of degree $-1$ conformal sesquilinear maps $\{ \rho_k :  \mathcal{V}^{\otimes k} \rightarrow \mathcal{V}[\lambda_1, \ldots, \lambda_{k-1}] \}_{k \geq 2}$ that satisfy the identities (\ref{par-a1}).

Let $(\mathcal{V} = \oplus_{i \in \mathbb{Z}} \mathcal{V}_i, \rho_1)$ and $(\mathcal{W} = \oplus_{i \in \mathbb{Z}} \mathcal{W}_i, \theta_1)$ be two chain complexes of $\mathbb{C}[\partial]$-modules. We say that the complex $(\mathcal{W} , \theta_1)$ is a {\em deformation retract} of $(\mathcal{V}, \rho_1)$ if there exist chain maps $p : (\mathcal{V}, \rho_1) \rightarrow (\mathcal{W}, \theta_1)$ and $i: (\mathcal{W}, \theta_1) \rightarrow  (\mathcal{V}, \rho_1)$ between complexes of $\mathbb{C}[\partial]$-modules and a degree $+1$ chain homotopy $h : \mathcal{V} \rightarrow \mathcal{V}$,
\begin{align}\label{contraction-data}
\xymatrix{
	  (\mathcal{V} , \rho_1)  \ar@(ul,ur)^h \ar@<2pt>[r]^{p} & ( \mathcal{W}, \theta_1) \ar@<2pt>[l]^{i} 
}
\end{align}
satisfying 
\begin{align*}
\mathrm{id}_\mathcal{V} - ip = \rho_1 h + h \rho_1 ~\text{  and }~ pi = \mathrm{id}_\mathcal{W}.
\end{align*}
In this case, the full data described in (\ref{contraction-data}) is often called contraction data.

Next, suppose that the complex $(\mathcal{V}, \rho_1)$ carries a degree $-1$ $\lambda$-multiplication $\rho_2 : \mathcal{V} \otimes \mathcal{V} \rightarrow \mathcal{V}[\lambda]$ which makes $(\mathcal{V}, \rho_1, \rho_2)$ into an $A_\infty[1]$-conformal algebra (here $\rho_k = 0$ for $k \geq 3$). In other words, $\rho_2$ makes the triple $(\mathcal{V}, \rho_1, \rho_2)$ into a shifted differential graded associative conformal algebra. We now define a degree $-1$ conformal sesquilinear map $\theta_2 : \mathcal{W} \otimes \mathcal{W} \rightarrow \mathcal{W}[\lambda ]$ by
\begin{align*}
    (\theta_2)_\lambda (x, y) := p  \big(  (\rho_2)_\lambda (i(x), i(y)) \big), \text{ for } x, y \in \mathcal{W}.
\end{align*}

Then we have the following straightforward observation.
\begin{lemma}
The map $\theta_2 \in \mathrm{Hom}_{cs}^{-1} (\mathcal{W}^{\otimes 2}, \mathcal{W}[\lambda])$ satisfies $\partial_{\theta_1} (\theta_2) = 0$. Moreover, for any $x, y, z \in \mathcal{W}$, we have
\begin{align}\label{obs}
    (\theta_2)_{\lambda + \mu} \big(  (\theta_2)_\lambda (x, y), z  \big) + (-1)^{|x|} (\theta_2)_\lambda \big( x, (\theta_2)_\mu (y, z)  \big) = - (\partial_{\theta_1} (\theta_3))_{\lambda, \mu} (x, y, z),
\end{align}
where $\theta_3 : \mathcal{W}^{\otimes 3} \rightarrow \mathcal{W} [\lambda, \mu]$ is the degree $-1$ conformal sesquilinear map given by
\begin{align*}
    (\theta_3)_{\lambda, \mu} (x, y, z) = p (\mu_2)_{\lambda + \mu}  \big(  h (\mu_2)_\lambda (i(x), i(y)), i(z)  \big) + (-1)^{|x|} p (\mu_2)_\lambda \big( i(x), h (\mu_2)_\lambda (i(y), i(z)) \big).
\end{align*}
\end{lemma}

It follows from the above lemma that $\theta_2$ need not satisfy the shifted conformal associativity condition (in general). That it, $(\mathcal{W}, \theta_1, \theta_2)$ need not be an $A_\infty [1]$-conformal algebra. However, it does satisfy if the maps $i$ and $p$ are inverses to each other. Further, the obstruction for the shifted conformal associativity of $\theta_2$ is given by $- \partial_{\theta_1} (\theta_3)$. Observe that the identity (\ref{obs}) is equivalent to the condition $\theta_2 \diamond \theta_2 = - \partial_{\theta_1} (\theta_3)$, which is nothing but the condition (\ref{par-a1}) for $n=3$. In fact, in the following, we will construct a sequence $\{ \theta_k : \mathcal{W}^{\otimes k} \rightarrow \mathcal{W} [\lambda_1, \ldots, \lambda_{k-1}] \}_{k \geq 3}$ of degree $-1$ conformal sesquilinear maps satisfying (\ref{par-a1}). Moreover, for $k=3$, the construction coincides with $\theta_3$ as defined in the above lemma. Our method for the constructions of higher $\theta_k$'s is similar to the classical case of $A_\infty$-algebras \cite{markl-a-inf,kadei}. For that, we first define a sequence $\{ \rho_k : \mathcal{V}^{\otimes k} \rightarrow \mathcal{V} [\lambda_1, \ldots, \lambda_{k-1}] \}_{k \geq 2}$ of maps inductively as follows. We define inductively as
\begin{align*}
    \rho_k = \sum_{ \substack{i+j = k\\ i, j \geq 1}} \rho_2 \big( (h \circ \rho_i) \otimes (h \circ \rho_j) \big)
\end{align*}
with the convention that $h \circ \rho_1 = \mathrm{id}$. Alternatively, the maps $\rho_k$'s can be defined in terms of the planar binary trees with $k$ leaves. For any planar binary tree $T$ with $k$ leaves ($k \geq 3$), there exist unique planar binary trees $T_1$ and $T_2$ such that $T = T_1 \vee T_2$, where $\vee$ stands for the grafting of trees. For such $T$, we associate a conformal sesquilinear map $\rho_T : \mathcal{V}^{\otimes k} \rightarrow \mathcal{V} [\lambda_1, \ldots, \lambda_{k-1}]$ by $\rho_T = \rho_2 \big( (h \circ \rho_{T_1}) \otimes (h \circ \rho_{T_2} ) \big)$ with the convention that $h \circ \rho_{ |} = \mathrm{id}$, where $|$ is the unique planar tree with one leaf. Finally, we define
\begin{align*}
    \rho_k := \sum_{T \in \mathsf{PBT}_{(k)}} \rho_T,
\end{align*}
where $\mathsf{PBT}_{(k)}$ is the set of all isomorphism classes of planar binary trees with $k$ leaves. Using the map $\rho_k$, we now define $\theta_k : \mathcal{W}^{\otimes k} \rightarrow \mathcal{W} [\lambda_1, \ldots, \lambda_{k-1}]$ by $\theta_k := p \circ \rho_k \circ i^{\otimes k}$, for $k \geq 3$.

\begin{thm} (Homotopy transfer theorem)
Let (\ref{contraction-data}) be a contraction data. If $(\mathcal{V}, \rho_1, \rho_2)$ is an $A_\infty [1]$-conformal algebra then $(\mathcal{W}, \{ \theta_k \}_{k \geq 1})$ is an $A_\infty [1]$-conformal algebra.
\end{thm}

The proof of the above theorem is similar to the classical case \cite{markl-a-inf,kadei}. The above homotopy transfer theorem can be further generalized. More precisely, we will now show that if the chain complex $(\mathcal{V}, \rho_1)$ inherits a full $A_\infty [1]$-conformal algebra structure, then it can be transferred to $(\mathcal{W}, \theta_1)$. In this case, one need to use planar trees (not necessarily binary). For any planar tree $T$ with $k$ leaves, there exist unique planar trees $T_1, \ldots, T_l$ such that $T = T_1 \vee \cdots \vee T_l$. Then we define $\rho_T = \rho_l \big(  (h \circ \rho_{T_1}) \otimes \cdots \otimes  (h \circ \rho_{T_l})   \big)$, and 
\begin{align*}
    \rho_k := \sum_{T \in \mathsf{PB}_{(k)}} \rho_T,
\end{align*}
where $\mathsf{PB}_{(k)}$ is the set of all isomorphism classes of planar trees with $k$ leaves. Finally, we define the map $\theta_k : \mathcal{W}^{\otimes k} \rightarrow \mathcal{W} [\lambda_1, \ldots, \lambda_{k-1}]$ $\theta_k := p \circ \rho_k \circ i^{\otimes k}$, for $k \geq 3$.

\begin{thm} (Homotopy transfer theorem) 
    Let (\ref{contraction-data}) be a contraction data. If $(\mathcal{V}, \{ \rho_k \}_{k \geq 1})$ is an $A_\infty [1]$-conformal algebra, then this structure transfers to an $A_\infty [1]$-conformal algebra structure on $(\mathcal{W}, \theta_1)$.
\end{thm}

Recall that an $A_\infty$-conformal algebra is equivalent to an $A_\infty [1]$-conformal algebra on the shifted graded $\mathbb{C}[\partial]$-module (Proposition \ref{proposition 3.6}). On the other hand, if there is a deformation retract of chain complexes, then it induces a deformation retract of the shifted chain complexes. Combining these results, we can now prove homotopy transfer theorem for $A_\infty$-conformal algebras.

\begin{thm} (Homotopy transfer theorem) Let $\xymatrix{
	  (\mathcal{A} , \mu_1)  \ar@(ul,ur)^h \ar@<2pt>[r]^{p} & ( \mathcal{H}, \nu_1) \ar@<2pt>[l]^{i} 
}$ be a given contraction data. If $(\mathcal{A}, \{ \mu_k \}_{k \geq 1})$ is an $A_\infty$-conformal algebra, then it can be transferred to an $A_\infty$-conformal algebra structure on $(\mathcal{H}, \nu_1)$.
\end{thm}

\section{Characterizations of some homotopy associative conformal algebras} \label{section-4}
In this section, we consider those $A_\infty$-conformal algebras whose underlying graded $\mathbb{C}[\partial]$-module is concentrated only in two degrees. In particular, we study  `skeletal' $A_\infty $-conformal algebras. We show that skeletal $A_\infty$-conformal algebras can be characterized by third cocycles of associative conformal algebras.
%and strict $A_\infty$-conformal algebras can be characterized by crossed modules of associative conformal algebras.

A $2$-term  $A_\infty $-conformal algebra is an $A_\infty $-conformal algebra $(\mathcal{A}, \{ \mu_k\}_{k \geq 1})$  whose underlying graded $\mathbb{C}[ \partial ]$-module $\mathcal{ A }$ is concentrated in degrees $0$ and $1$, that is, $\mathcal{ A }= \mathcal{ A}_0 \oplus \mathcal{ A}_1$. In that case, we have $\mu_k=0 $ for $ k >3 $. If we write explicitly the higher associative conformal identities (\ref{a-inf-conf}), we obtain the following. 
  
  \begin{definition}\label{definition 4.1}
       A {\em $2$-term $A_\infty$-conformal algebra}  is a graded $\mathbb{C}[ \partial ]$-module $\mathcal{ A }= \mathcal{ A}_0 \oplus \mathcal{A}_{1}$ equipped with
   
 - a  $\mathbb{C}[\partial]$-module homomorphism $\beta: \mathcal{ A}_{1} \to \mathcal{ A}_{0}$,
 
 - a $\mathbb{C}$-linear conformal sesquilinear map $\mu_{2} : \mathcal{A}_{i} \otimes \mathcal{A}_{j} \to \mathcal{A}_{i+j}[\lambda]$, for $0 \leq i, j, i+j \leq 1$,
 
- a $\mathbb{C}$-linear conformal sesquilinear map $\mu_{3} : \mathcal{A}_{0} \otimes \mathcal{A}_{0} \otimes \mathcal{A}_{0} \to \mathcal{A}_{1}[\lambda_{1},\lambda_{2}]$ subject to the following conditions:
 \begin{enumerate}
 \item[(i)] $(\mu_{2})_{\lambda}(m,n) = 0$,
     \item[(ii)] $ \beta ((\mu_{2})_{\lambda}(a,m)) = (\mu_{2})_{\lambda}(a, \beta m)$,
     \item[(iii)] $\beta ((\mu_{2})_{\lambda}(m,a)) = (\mu_{2})_{\lambda}(\beta m,a)$,
     \item[(iv)] $(\mu_{2})_{\lambda}(\beta m,n) = (\mu_{2})_{\lambda}(m, \beta n)$,
      \item[(v)] $\beta ((\mu_{3})_{\lambda_{1},\lambda_{2}}(a,b,c)) = (\mu_{2})_{\lambda_{1}+\lambda_{2}}((\mu_{2})_{\lambda_{1}}(a,b),c)-(\mu_{2})_{\lambda_{1}}(a,(\mu_{2})_{\lambda_{2}}(b,c))$\label{equation 5},
     \item[(vi)] $(\mu_{3})_{\lambda_{1},\lambda_{2}}(a,b, \beta m)= (\mu_{2})_{\lambda_{1}+\lambda_{2}}((\mu_{2})_{\lambda_{1}}(a,b),m)-(\mu_{2})_{\lambda_{1}}(a,(\mu_{2})_{\lambda_{2}}(b,m))$\label{equation 6},
     \item[(vii)] $(\mu_{3})_{\lambda_{1},\lambda_{2}}(a, \beta m,c)= (\mu_{2})_{\lambda_{1}+\lambda_{2}}((\mu_{2})_{\lambda_{1}}(a,m),c)-(\mu_{2})_{\lambda_{1}}(a,(\mu_{2})_{\lambda_{1}}(m,c))$\label{equation 7},
     \item[(viii)] $(\mu_{3})_{\lambda_{1},\lambda_{2}}(\beta m,b,c)= (\mu_{2})_{\lambda_{1}+\lambda_{2}}((\mu_{2})_{\lambda_{1}}(m,b),c)-(\mu_{2})_{\lambda_{1}}(m,(\mu_{2})_{\lambda_{1}}(b,c))$\label{equation 8},
     \item[(ix)] $(\mu_{3})_{\lambda_{1}+\lambda_{2},\lambda_{3}}((\mu_{2})_{\lambda_{1}}(a,b),c,d) - (\mu_{3})_{\lambda_{1},\lambda_{2}+\lambda_{3}}(a,(\mu_{2})_{\lambda_{2}}(b,c),d)+ (\mu_{3})_{\lambda_{1},\lambda_{2}}(a,b,(\mu_{2})_{\lambda_{3}}(c,d))\\
     \quad =(\mu_{2})_{\lambda_{1}+\lambda_{2}+\lambda_{3}}((\mu_{3})_{\lambda_{1},\lambda_{2}}(a,b,c),d)+(\mu_{2})_{\lambda_{1}}(a,(\mu_{3})_{\lambda_{2},\lambda_{3}}(b,c,d))$\label{equation 9},
 \end{enumerate}
  for all $a,b,c,d \in \mathcal{A}_{0}$ and $m,n \in \mathcal{A}_{1}$.
 We denote a $2$-term $A_{\infty}$-conformal algebra as above simply by $(\mathcal{A}_{1}\xrightarrow{\beta}\mathcal{A}_{0},\mu_{2},\mu_{3})$.
 \end{definition}
  
  Let $(\mathcal{A}_1 \xrightarrow{\beta} \mathcal{A}_0, \mu_2, \mu_3)$ and $(\mathcal{A}'_1 \xrightarrow{\beta'} \mathcal{A}'_0, \mu'_2, \mu'_3)$ be two $2$-term $A_\infty$-conformal algebras. A {\em homomorphism} of $2$-term $A_\infty$-conformal algebras from $(\mathcal{A}_1 \xrightarrow{\beta} \mathcal{A}_0, \mu_2, \mu_3)$ to $(\mathcal{A}'_1 \xrightarrow{\beta'} \mathcal{A}'_0, \mu'_2, \mu'_3)$ is given by a triple $(f_0, f_1, f_2)$ consisting of $\mathbb{C}[\partial]$-module maps $f_0 : \mathcal{A}_0 \rightarrow \mathcal{A}'_0$ and $f_1 : \mathcal{A}_1 \rightarrow \mathcal{A}'_1$, and a conformal sesquilinear map $f_2 : \mathcal{A}_0 \otimes \mathcal{A}_0 \rightarrow \mathcal{A}_1' [\lambda],$ $a \otimes b \mapsto (f_2)_\lambda (a, b)$ such that the following identities are hold: for any $a, b, c \in \mathcal{A}_0$ and $m \in \mathcal{A}_1$,
  \begin{align*}
      \beta' \big(  (f_2)_\lambda (a, b) \big) =~& f_0 \big(  (\mu_2)_\lambda (a, b) \big) - (\mu'_2)_\lambda \big( f_0 (a), f_0 (b)  \big), \\
      (f_2)_\lambda (a, \beta m) =~& f_1 \big(  (\mu_2)_\lambda (a, m) \big) - (\mu'_2)_\lambda \big(  f_0 (a), f_1 (m) \big), \\
      (f_2)_\lambda (\beta m, a) =~& f_1 \big( (\mu_2)_\lambda (m,a) \big) - (\mu'_2)_\lambda \big( f_1(m), f_0 (a) \big), \\
      (f_2)_\lambda \big( a, (\mu_2)_\mu (b, c) \big) -& (f_2)_{\lambda + \mu} \big(  (\mu_2)_\lambda (a, b), c \big) + (\mu'_2)_{\lambda + \mu} \big( (f_2)_{\lambda} (a, b), f_0 (c) \big)  - (\mu_2')_\lambda \big( f_0 (a), (\mu_2)_\mu (b, c)  \big) \\
      =~& (\mu'_3)_{\lambda, \mu} \big(  f_0 (a), f_0 (b), f_0 (c) \big) - f_1 \big( (\mu_3)_{\lambda, \mu} (a, b, c)  \big).
  \end{align*}

  Suppose $f = (f_0, f_1, f_2)$ is a homomorphism of $2$-term $A_\infty$-conformal algebras from $(\mathcal{A}_1 \xrightarrow{\beta} \mathcal{A}_0, \mu_2, \mu_3)$ to $(\mathcal{A}'_1 \xrightarrow{\beta'} \mathcal{A}'_0, \mu'_2, \mu'_3)$, and $g = (g_0, g_1, g_2)$ is a homomorphism from $2$-term $A_\infty$-conformal algebras from $(\mathcal{A}'_1 \xrightarrow{\beta'} \mathcal{A}'_0, \mu'_2, \mu'_3)$ to $(\mathcal{A}''_1 \xrightarrow{\beta''} \mathcal{A}''_0, \mu''_2, \mu''_3)$. Their composition is the homomorphism 
  \begin{align*}
  (g \circ f) :=~& \big( g_0 \circ f_0 , g_1 \circ f_1, (g \circ f)_2 \big), \text{ where } \\
  \big(  (g \circ f)_2  \big)_\lambda (a, b) =~& (g_2)_\lambda \big(  f_0 (a), f_0 (b)  \big) + g_1 \big(   (f_2)_\lambda (a, b) \big), \text{ for } a, b \in \mathcal{A}_0.
\end{align*}
The collection of all  $2$-term $A_{\infty}$-conformal algebras and homomorphisms between them form a category, denoted by ${\bf 2A_{\infty}conf}$.

  \medskip

  \medskip
  
\noindent {\bf Skeletal $A_{\infty}$-conformal algebras.}
Here we focus on skeletal $A_{\infty}$-conformal algebras which are a particular type of $2$-term $A_{\infty}$-conformal algebras. We also characterize skeletal $A_{\infty}$-conformal algebras by third cocycles of associative conformal algebras.

\begin{definition}
    A $2$-term $A_{\infty}$-conformal algebra $(\mathcal{A}_{1}\xrightarrow{ \beta }\mathcal{A}_{0},\mu_{2},\mu_{3})$ is called {\em skeletal} if $\beta =0$.
\end{definition}

Let $(\mathcal{A}_{1}\xrightarrow{0}\mathcal{A}_{0},\mu_{2},\mu_{3})$ and $(\mathcal{A}_{1}\xrightarrow{0}\mathcal{A}_{0},\mu_{2}^{\prime},\mu_{3}^{\prime})$ be two skeletal $A_{\infty}$-conformal algebras with the same underlying graded $\mathbb{C}[\partial]$-module. They are said to be {\em equivalent} if $\mu_{2} = \mu_{2}^{\prime}$ and there exists a conformal sesquilinear map $\sigma : \mathcal{A}_{0}\otimes\mathcal{A}_{0} \to \mathcal{A}_{1}[\lambda]$ such that 
  \begin{align}
   \label{equiv-co}
      (\mu_{3})^{\prime}_{\lambda_{1},\lambda_{2}}(a,b,c) = ~&(\mu_{3})_{\lambda_{1},\lambda_{2}}(a,b,c) +(\mu_{2})_{\lambda_{1}}(a,\sigma_{\lambda_{2}}(b,c)) - \sigma_{\lambda_{1}+\lambda_{2}}((\mu_{2})_{\lambda_{1}}(a,b),c)\\&
       -\sigma_{\lambda_{1}}(a,(\mu_{2})_{\lambda_{2}}(b,c))-(\mu_{2})_{\lambda_{1}+\lambda_{2}}(\sigma_{\lambda_{1}}(a,b),c), \nonumber
  \end{align}
  for all $a,b,c \in \mathcal{A}_{0}$.  With these, we can characterize (equivalence classes of) skeletal $A_{\infty}$-conformal algebras.
 
  \begin{thm}\label{Theorem 2}
      There is a one-to-one correspondence between skeletal $A_{\infty}$-conformal algebras and tuples $(A,M,\Theta)$ in which $A$ is an associative conformal algebra, $M$ is a conformal $A$-bimodule and $\Theta \in Z^{3}(A,M)$ is a $3$-cocycle.
      
      Moreover, this extends to a one-to-one correspondence between the equivalence classes of skeletal $A_{\infty}$-conformal algebras and tuples $(A,M,[\Theta])$, where $[\Theta] \in H^{3}(A,M)$.
  \end{thm}

\medskip
  
  %The above result has a generalization when we consider the higher cohomology groups (instead of $H^3$).

  We will prove a more general result for which the above theorem is a particular case.  An {\em $n$-term  $A_{\infty}$-conformal algebra} is an $A_{\infty}$-conformal algebra $(\mathcal{A},\{\mu_{k}\}_{ k \geq 1})$ whose uderlying graded $\mathbb{C}[\partial]$-module $\mathcal{A}$ is concentrated in degrees $0,1,\ldots,n-1$. That is,
  \begin{align*}
       \mathcal{A} = \mathcal{A}_{0}\oplus\mathcal{A}_{1}\oplus\cdots\oplus\mathcal{A}_{n-1}.
  \end{align*}
  An $n$-term $A_{\infty}$-conformal algebra $(\mathcal{A}_{0}\oplus\mathcal{A}_{1}\oplus\cdots\oplus\mathcal{A}_{n-1}, \{\mu_{k}\}_{k \geq 1})$ is said to be {\em skeletal} 
 if it that has only two nonzero terms $\mathcal{A}_{0}$ and $\mathcal{A}_{n-1}$ with the zero differential $(\mathrm{i.e.}~ \mu_{1} = 0)$. For the degree reason, the maps $\mu_{k}$'s are nonzero only for $k = 2$ and $n+1$. It also turns out from the identities (\ref{a-inf-conf}) that $\mathcal{A}_0$ is an associative conformal algebra with the $\lambda$-multiplication $a_\lambda b := (\mu_2)_\lambda (a, b)$, for $a, b \in \mathcal{A}_0$. Moreover, $\mathcal{A}_{n-1}$ is a conformal $\mathcal{A}_0$-bimodule with the left and right $\lambda$-actions are respectively given by
 \begin{align*}
     a_\lambda m := (\mu_2)_\lambda (a, m) ~~~ \text{ and } ~~~ m_\lambda a := (\mu_2)_\lambda (m, a), \text{ for } a \in \mathcal{A}_0, m \in \mathcal{A}_{n-1}. 
 \end{align*}
 A skeletal $n$-term $A_{\infty}$-conformal algebra may be simply denoted by $(\mathcal{A}_{n-1}\xrightarrow{0} 0 \to \cdots \xrightarrow{0} \mathcal{A}_{0}, \mu_{2},\mu_{n+1})$.
  
  Let $(\mathcal{A}_{n-1}\xrightarrow{0} 0 \to \cdots \xrightarrow{0} \mathcal{A}_{0}, \mu_{2},\mu_{n+1})$ and  $(\mathcal{A}_{n-1}\xrightarrow{0} 0 \to \cdots \xrightarrow{0} \mathcal{A}_{0}, \mu'_{2},\mu'_{n+1})$ be two skeletal $n$-term  $A_{\infty}$-conformal algebras with the same underlying graded $\mathbb{C}[\partial]$-module. They are said to be {\em equivalent} if $\mu_{2}=\mu_{2}^{\prime}$ and there exists a conformal sesquilinear map $\sigma :( \mathcal{A}_{0})^{\otimes n } \to \mathcal{A}_{n-1}[\lambda_{1},\ldots,\lambda_{n-1}]$ that satisfies 
  \begin{align}\label{equivv}
      (\mu_{n+1}^{\prime})_{\lambda_{1},\ldots,\lambda_{n}}(a_{1},\ldots,a_{n+1}) = (\mu_{n+1})_{\lambda_{1},\ldots,\lambda_{n}}(a_{1},\ldots,a_{n+1}) + (\delta\sigma)_{\lambda_{1},\ldots,\lambda_{n}}(a_{1},\ldots,a_{n+1}), 
  \end{align}
  for all $a_{1},\ldots,a_{n+1} \in \mathcal{A}_{0}$. Here $\delta$ is the Hochschild coboundary operator of the associative conformal algebra $\mathcal{A}_0$ with coefficients in the conformal $\mathcal{A}_0$-bimodule $\mathcal{A}_{n-1}$. With the above notations, we have the following generalization of Theorem \ref{Theorem 2}.
  
  \begin{thm}
      There is a one-to-one correspondence between skeletal $n$-term $A_{\infty}$-conformal algebras and tuples $(A,M,\Theta)$ in which $A$ is an associative conformal algebra, $M$ is a conformal $A$-bimodule and $\Theta \in Z^{n+1}(A,M)$ is a $(n+1)$-cocycle of the associative conformal algebra $A$ with coefficients in the conformal $A$-bimodule $M$. Further, this extends to a one-to-one correspondence between the equivalence classes of skeletal $n$-term $A_{\infty}$-conformal algebras and tuples $(A,M,[\Theta])$, where $[\Theta] \in H^{n+1}(A,M)$.
  \end{thm}
  
   \begin{proof}
      Let $(\mathcal{A}_{n-1}\xrightarrow{0} 0 \to \cdots \xrightarrow{0} \mathcal{A}_{0}, \mu_{2},\mu_{n+1})$ be a skeletal $n$-term $A_{\infty}$-conformal algebra. Then it follows from the definition of an $A_{\infty}$-conformal algebra that the conformal sesquilinear map $\mu_{2} : \mathcal{A}_{0}\otimes\mathcal{A}_{0} \to\mathcal{A}_{0}[\lambda]$ makes the $\mathbb{C}[\partial]$-module $\mathcal{A}_{0}$ into an associative conformal algebra. Moreover, the conformal sesquilinear maps $\mu_{2} : \mathcal{A}_{0}\otimes\mathcal{A}_{n-1} \to\mathcal{A}_{n-1}[\lambda] $ and $\mu_{2} : \mathcal{A}_{n-1}\otimes\mathcal{A}_{0} \to\mathcal{A}_{n-1}[\lambda]$ makes the $\mathbb{C}[\partial]$-module $\mathcal{A}_{n-1}$ into a conformal $\mathcal{A}_{0}$-bimodule. Further, it follows from the higher associative conformal identities $(\ref{a-inf-conf})$ that 
      \begin{align}\label{equation-n}
     &   \sum_{k+l=n+3}\sum_{i=1}^{k}(-1)^{i(l+1) +l(|a_{1}|+\cdots+|a_{i-1}|)}\\&
     (\mu_{k})_{\lambda_{1},\ldots,\lambda_{i-1},\lambda_{i}+\cdots+\lambda_{i+l-1},\ldots,\lambda_{n+1}}(a_{1},\ldots,a_{i-1},(\mu_{l})_{\lambda_{i},\ldots,\lambda_{i+l-2}}(a_{i},\ldots,a_{i+l-1}),a_{i+l},\ldots,a_{n+2})= 0, \nonumber
    \end{align}
    for all $a_1, \ldots, a_{n+2} \in \mathcal{A}_0$.
   The non-zero terms in the above summation occur when $k = 2$ (then $l=n+1$) and $k=n+1$ (then $l=2$). Thus, $(\ref{equation-n})$ is equivalent to 
   \begin{align*}
   (-1)^{n+2}& (\mu_{2})_{\lambda_{1}+\cdots+\lambda_{n+1}}\big( (\mu_{n+1})_{\lambda_{1},\ldots,\lambda_{n}}(a_{1},\ldots,a_{n+1}),a_{n+2} \big) + (\mu_{2})_{\lambda_{1}}\big( a_{1},(\mu_{n+1})_{\lambda_{1},\ldots,\lambda_{n+1}}(a_{2},\ldots,a_{n+2}) \big)\\&
 + \sum_{i=1}^{n+1}(-1)^{3i}(\mu_{n+1})_{\lambda_{1},\ldots,\lambda_{i-1},\lambda_{i}+\lambda_{i+l},\ldots,\lambda_{n+1}}\big( a_{1},\ldots,a_{i-1},(\mu_{2})_{\lambda_{i}}(a_{i},a_{i+1}),\ldots,a_{n+2} \big)= 0.
   \end{align*}
  This is same as $(\delta\mu_{n+1})_{\lambda_{1},\ldots,\lambda_{n+1}}(a_{1},\ldots,a_{n+2}) = 0$. In other words, $\mu_{n+1} \in Z^{n+1}(\mathcal{A}_{0},\mathcal{A}_{n-1})$ is a $(n+1)$-cocycle of the associative conformal algebra $\mathcal{A}_0$ with coefficients in the conformal $\mathcal{A}_0$-bimodule $\mathcal{A}_{n-1}$.

  Conversely, let $(A,M,\Theta)$ be a triple consisting of an associative conformal algebra $A$, a conformal $A$-bimodule $M$ and a $(n+1)$-cocycle $\Theta \in Z^{n+1}(A,M)$. Then it is straightforward to verify that the triple $(M\xrightarrow{0} 0 \rightarrow \cdots \xrightarrow{0} A ,\mu_{2},\mu_{n+1}=\Theta)$ is a skeletal $n$-term $A_{\infty}$-conformal algebra, where $\mu_{2}$ is given by the $\lambda$-multiplication of $A$, and left and right $\lambda$-actions of $A$ on $M$.

  The final part about the correspondence between equivalence classes of skeletal $n$-term $A_{\infty}$-conformal algebras and tuples $(A, M,[\Theta])$ follows from the identity (\ref{equivv}).
  \end{proof}

\medskip

\medskip

\section{Associative conformal 2-algebras } \label{section-5}
In this section, we introduce the notion of associative conformal $2$-algebras. They can be thought of as the categorification of {associative conformal algebras}. Our main result in this section shows that the category of associative conformal $2$-algebras and the category of $2$-term $A_{\infty}$-conformal algebras are equivalent.

Let ${\bf {Vect}^{ \mathbb{C} [\partial]}}$ be the category of all $\mathbb{C} [\partial]$-modules and $\mathbb{C}[\partial]$-linear maps.  A {\em $2$-vector space in the category of $\mathbb{C}[\partial]$-modules} is a category internal to the category  ${\bf {Vect}^{ \mathbb{C} [\partial]}}$.
Thus, it follows that a $2$-vector space in the category of $\mathbb{C}[\partial]$-modules is a category $C=(C_{1} \rightrightarrows C_{0})$ with $\mathbb{C}[\partial]$-module of objects $C_{0}$ and $\mathbb{C}[\partial]$-module of morphisms $C_{1}$ such that the source and target maps $s,t : C_{1} \to C_{0}$, the identity-assigning map $i : C_{0} \to C_{1}$ and the composition map $C_{1}\times_{ C_{0}} C_{1}\to C_{1}$ are all $\mathbb{C}[\partial]$-linear maps. We write the images of the object-inclusion map as $i (a) = 1_a$, for $a \in A$. A {\em morphism} between $2$-vector spaces in the category of $\mathbb{C}[\partial]$-modules is a functor internal to the category ${\bf {Vect}^{ \mathbb{C}[\partial]}}$. Thus, a morphism $F : C \rightarrow C'$ between $2$-vector spaces in the category of $\mathbb{C}[\partial]$-modules is given by a pair $F=(F_0, F_1)$ of $\mathbb{C}[\partial]$-linear maps $F_0 : C_0 \rightarrow C'_0$ and $F_1 : C_1 \rightarrow C'_1$ commuting with all the structure maps.

The notion of $2$-vector spaces in the category of $\mathbb{C}[\partial]$-modules are closely related to $2$-term complexes of $\mathbb{C}[\partial]$-modules. Let $C=(C_{1}\rightrightarrows C_{0})$ be a $2$-vector space in the category of $\mathbb{C}[\partial]$-modules. Then $\mathrm{ker}(s)\xrightarrow{t} C_{0} $ is a $2$-term complex of $\mathbb{C}[\partial]$-modules. On the other hand, if $\mathcal{A}_{1} \xrightarrow{ \beta} \mathcal{A}_{0}$ is a complex of $\mathbb{C}[\partial]$-modules then
\begin{align}\label{2-to-2}
    C = ( \mathcal{A}_{0}\oplus \mathcal{A}_{1}\rightrightarrows \mathcal{A}_{0})
\end{align}
is a $2$-vector space in the category of $\mathbb{C}[\partial]$-modules, where the source, target and object inclusion are respectively given by 
\begin{align*}
    s(a,h) =a, \quad t(a,h)= a+ \beta h ~~ \mathrm{~and~} ~~ i(a) = (a,0).
\end{align*}
A morphism between $2$-term complexes induces a morphism between the corresponding $2$-vector spaces in the category of $\mathbb{C}[\partial]$-modules.

\begin{definition}
    An {\em associative conformal $2$-algebra} is a triple $(C,\pi,\mathbb{A})$ consisting of a $2$-vector space $C$ in the category of $\mathbb{C}[\partial]$-modules with a conformal sesquilinear functor $\pi :C\otimes C 
 \to C[\lambda]$, $a\otimes b \mapsto a _{\lambda}b$ and a conformal sesquilinear natural isomorphism (called the conformal associator)
 \begin{align*}
     \mathbb{A}_{a,b,c} : a_{\lambda}(b_{\mu}c)\longrightarrow (a_{\lambda}b)_{\lambda + \mu}c
 \end{align*}
 satisfying the following pentagonal identity:

\[
%\xymatrixrowsep{0.1in}
\xymatrixcolsep{0.1cm}
\xymatrix{
& & a_\lambda (b_\mu (c_\nu d)) \ar[lldd]_{ \mathbb{A}_{b,c,d} } \ar[rrdd]^{\mathbb{A}_{a. b, c_\nu d}} & &  \\
 & & & &  \\
a_\lambda ( (b_\mu c)_{\mu + \nu} d ) \ar[rdd]_{\mathbb{A}_{a, b_\mu c, d}} & & & & (a_\lambda b)_{\lambda + \mu} (c_\nu d) \ar[ldd]_{\mathbb{A}_{a_\lambda b, c , d}} \\
 & & & & \\
 & (a_\lambda (b_\mu c))_{\lambda + \mu + \nu} d \ar[rr]_{\mathbb{A}_{a, b, c}} & & ((a_\lambda b)_{\lambda + \mu} c )_{\lambda + \mu + \nu} d & 
}
\]
for all $a,b,c,d \in C_0$.
\end{definition}

\begin{definition}
Let $(C, \pi, \mathbb{A})$ and $(C', \pi', \mathbb{A}')$ be two associative conformal $2$-algebras. A {\em homomorphism} of associative conformal $2$-algebras from $(C, \pi, \mathbb{A})$ to $(C', \pi', \mathbb{A}')$ is a functor $F= (F_0 , F_1) : C \rightarrow C'$ between the underlying $2$-vector spaces in the category of $\mathbb{C}[\partial]$-modules and a conformal sesquilinear natural isomorphism
\begin{align*}
    \mathbb{F}_{a, b} : F_0(a)'_\lambda F_0(b) \rightarrow F_0 (a_\lambda b)
\end{align*}
that makes the following diagram commutative:
\[
\xymatrix{
 &  F_0(a)'_\lambda (F_0(b)'_\mu F_0(c)) \ar[ld]_{ \mathbb{A}'_{F(a), F(b), F(c)}} \ar[rd]^{   \mathbb{F}_{b, c} } &   \\
 (F_0(a)'_\lambda F_0(b))'_{\lambda + \mu} F_0(c) \ar[d]_{\mathbb{F}_{a, b}} & & F_0(a)'_\lambda F_0(b_\mu c) \ar[d]^{ \mathbb{F}_{a, b_\mu c}} \\
 F_0 (a_\lambda b)'_{\lambda + \mu} F_0 (c) \ar[rd]_{  \mathbb{F}_{a_\lambda b, c}} & & F_0 (a_\lambda (b_\mu c)) \ar[ld]^{\mathbb{A}_{a, b, c}} \\
 & F_0 ( (a_\lambda b)_{\lambda + \mu} c) &  
}
\]
for all $a, b, c \in C_0$. We denote a homomorphism as above simply by the pair $(F= (F_0, F_1), \mathbb{F})$.
\end{definition} 

Let $(C,\pi,\mathbb{A})$, $(C^{\prime},\pi^{\prime},\mathbb{A}^{\prime})$ and $(C^{\prime \prime},\pi^{\prime \prime},\mathbb{A}^{\prime \prime})$ be three associative conformal $2$-algebras. Suppose $(F = (F_0, F_1),\mathbb{F}) : (C,\pi,\mathbb{A}) \to (C^{\prime},\pi^{\prime},\mathbb{A}^{\prime})$  and $(G = (G_0, G_1), \mathbb{G}) : (C^{\prime},\pi^{\prime},\mathbb{A}^{\prime}) \to (C^{\prime \prime},\pi^{\prime \prime},\mathbb{A}^{\prime \prime})$ are homomorphisms between associative conformal $2$-algebras. Then their composition is given by $ (G \circ F = (G_0 \circ F_0, G_1 \circ F_1), \mathbb{G} \diamond \mathbb{F})$, where $\mathbb{G} \diamond \mathbb{F}$ is the ordinary composition
\begin{align*}
    (G_0 \circ F_0)(a)''_{\lambda}(G_0 \circ F_0)(b)\xrightarrow{\mathbb{G}_{F_0(a),F_0 (b)}}G_0 (F_0(a)'_{\lambda}F_0(b))\xrightarrow{\mathbb{F}_{a,b}} (G_0\circ F_0)(a_{\lambda} b), \text{ for } a, b \in C_0.
\end{align*}

For any associative conformal $2$-algebra $(C,\pi,\mathbb{A})$, the identity homomorphism is given by the identity functor $\mathrm{id}_{C}:C\to C$ with the identity natural isomorphism $\mathbbm{id}_{a, b} : a_\lambda b \rightarrow a_\lambda b$, for all $a, b \in C_0$. With the above definitions and notations, associative conformal $2$-algebras and homomorphisms between them form a category. We denote this category by {\bf conf2Alg}. We will now prove the main result of the present section. 

\begin{thm}
    The categories ${\bf 2A_{\infty}conf}$ and {\bf conf2Alg} are equivalent.
\end{thm}

\begin{proof}
    Here we only sketch the proof. Let $(\mathcal{A}_{1}\xrightarrow{ \beta }\mathcal{A}_{0},\mu_{2},\mu_{3})$ be a $2$-term $A_{\infty}$-conformal algebra. We consider the $2$-vector space $C=( \mathcal{A}_{0} \oplus \mathcal{A}_{1}\rightrightarrows \mathcal{A}_{0})$ in the category of $\mathbb{C}[\partial]$-modules given in (\ref{2-to-2}). We now define a conformal sesquilinear functor $\pi : C\otimes C \to C[\lambda], (a, h) \otimes (b, k) \mapsto (a,h)_\lambda (b, k)$ by 
    \begin{align*}
        (a,h)_{\lambda}(b,k) = ((\mu_{2})_{\lambda}(a,b),(\mu_{2})_{\lambda}(a,k)+ (\mu_{2})_{\lambda}(h,b) + (\mu_{2})_{\lambda}(dh,k)),
    \end{align*}
    for $(a,h), (b,k) \in \mathcal{A}_{0} \oplus \mathcal{A}_{1} $. We define the conformal associator by
    \begin{align}\label{conformal-associator}
        \mathbb{A}_{a,b,c} = \big( a_{\lambda}(b_{\mu}c),(\mu_{3})_{\lambda,\mu}(a,b,c) \big),\mathrm{~for~}a,b,c\in A_{0}.
    \end{align}
    Then it follows from the condition (v) of Definition \ref{definition 4.1} that $\mathbb{A}_{a,b,c}$ is a morphism from source $a_{\lambda}(b_{\mu}c)$ to target $(a_{\lambda}b)_{\lambda+\mu}c$. It is easy to check that $\mathcal{A}_{a,b,c}$ is a natural isomorphism. Finally, a direct computation shows that the conformal associator (\ref{conformal-associator}) satisfies the pentagonal identity. Thus, we obtain an associative conformal $ 2$-algebra $( C = ( \mathcal{A}_{0} \oplus \mathcal{A}_{1}\rightrightarrows \mathcal{A}_{0}), \pi, \mathbb{A})$.

    Let $(\mathcal{A}_{1}\xrightarrow{ \beta }\mathcal{A}_{0},\mu_{2},\mu_{3})$ and $(\mathcal{A}_{1}^{\prime}\xrightarrow{\beta^{\prime}}\mathcal{A}_{0}^{\prime},\mu_{2}^{\prime},\mu_{3}^{\prime})$ be $2$-term $A_{\infty}$-conformal algebras and $(f_{0},f_{1 },f_{2})$ be a homomorphism between them. Let $(C=(\mathcal{A}_{0} \oplus \mathcal{A}_{1} \rightrightarrows \mathcal{A}_{0}), \pi,\mathbb{A})$ and $(C^{\prime}=(\mathcal{A}_{0}^{\prime}\oplus \mathcal{A}_{1}^{\prime} \rightrightarrows \mathcal{A}_{0}^{\prime}), \pi^{\prime},\mathbb{A}^{\prime})$ be the corresponding associative conformal $2$-algebras. We define a functor $F : C \to C^{\prime}$ by $F = (F_0, F_1) := (f_0 , f_0 \oplus f_1)$ and a natural isomorphism  $\mathbb{F}_{a,b}: F_0(a)'_{\lambda} F_0(b) \to F_0(a _{\lambda}b)$ by $\mathbb{F}_{a,b} = (f_{0}(a)_{\lambda}f_{0}(b), (f_{2})_\lambda (a,b)) $, for all $a, b \in \mathcal{A}_0$.
    Then it is easy to check that $(F = (F_0, F_1), \mathbb{F})$ is a homomorphism of associative conformal $2$-algebras. The above constructions gives rise to a functor $S: {\bf 2A_{\infty}conf} \to {\bf conf2Alg}$.
    
    Conversely, let $(C=(C_{1}\rightrightarrows C_{0}), \pi,\mathbb{A})$ be an associative conformal $2$-algebra. We consider the two term chain complex $\mathcal{A}_{1} = \mathrm{ker}(s) \xrightarrow{t} C_{0}= \mathcal{A}_{0}$ and define conformal sesquilinear maps $\mu_{2} : \mathcal{A}_{i}\otimes \mathcal{A}_{j} \to \mathcal{A}_{i+j}[\lambda]$ and $\mu_{3} : \mathcal{A}_{0}\otimes \mathcal{A}_{0}\otimes \mathcal{A}_{0} \to \mathcal{A}_{1}[\lambda,\mu]$ by 
    \begin{align*}
       & (\mu_{2})_{\lambda}(a,b)= a_{\lambda}b,~ (\mu_{2})_{\lambda}(a,h)= {1_{a}}_{\lambda}h,~ (\mu_{2})_{\lambda}(h,a) = h_{\lambda}{1_{a}}\\&
       (\mu_{3})_{\lambda , \mu}(a,b,c) = \mathrm{pr}_{2} \big( \mathcal{A}_{a,b,c} : a_{\lambda}(b_{\mu}c) \to (a_{\lambda}b)_{\lambda +\mu}c  \big),
    \end{align*}
   where $\mathrm{pr}_{2} : C_{0}\oplus C_{1} \to C_{1}$ is the projection onto the second factor. Then one can easily verify that $(\mathcal{A}_1 = \mathrm{ker}(s) \xrightarrow{t} C_{0} = \mathcal{A}_0, \mu_{2},\mu_{3})$ is a $2$-term $A_{\infty}$-conformal algebra.
Let $(C,\pi,\mathbb{A})$ and  $(C^{\prime},\pi^{\prime},\mathbb{A}^{\prime})$ be two associative conformal $2$-algebras and $(F = (F_0, F_1), \mathbb{F})$ be a homomorphism between them. Let $(\mathrm{ker}(s) \xrightarrow{t} C_{0}, \mu_{2},\mu_{3})$ and $(\mathrm{ker}(s^{\prime}) \xrightarrow{t^{\prime}} C_{0}^{\prime}, \mu_{2}^{\prime},\mu_{3}^{\prime})$ be the corresponding $2$-term $A_{\infty}$-conformal algebras. We define $\mathbb{C}[\partial]$-linear maps $f_{0}: C_{0} \to C_{0}^{\prime}$, $f_{1} : \mathrm{ker}(s) \to \mathrm{ker}(s^{\prime})$ and a conformal sesquilinear map $f_{2} : C_{0}\otimes C_{0} \to  \mathrm{ker}(s^{\prime})[\lambda]$ by $f_{0}=F_{0}$,~ $f_{1} = F_1\vline_{ \mathrm{ ker}(s)}$ and
 $f_{2}(a,b) = \mathbb{F}_{a,b} - 1_{ s( \mathbb{F}_{ a,b })}$, for $a, b \in C_0 = \mathcal{A}_0$.
 It is easy to see that $(f_{0},f_{1},f_{2})$ gives rise to a morphism between $2$-term $A_{\infty }$-conformal algebras. Thus, we obtain a functor $T : {\bf conf2Alg} \to {\bf 2A_{\infty}conf}$.
 
% Finally, it is not hard to show that there are natural isomorphisms $\Lambda: T \circ S \implies 1_{ {\bf 2A_{\infty}conf} }$ and $\Upsilon: S \circ T\implies 1_{{\bf conf2Alg} }$. The constructions of $\Lambda$ and $\Upsilon$ are similar to those given in \cite{baez}. This completes the proof.

We are now left to show that there are natural isomorphisms $\Lambda: T \circ S \implies 1_{ {\bf 2A_{\infty}conf} }$ and $\Upsilon: S \circ T \implies 1_{{\bf conf2Alg} }$. Let $(\mathcal{A}_1 \xrightarrow{ \beta} \mathcal{A}_0, \mu_2, \mu_3)$ be a given $2$-term $A_\infty$-conformal algebra. If we apply $T \circ S$ to this $2$-term $A_\infty$-conformal algebra, we obtain the same $2$-term $A_\infty$-conformal algebra. Thus, $T \circ S$ is the identity functor $1_{ {\bf 2A_{\infty}conf} }$ on the category $ {\bf 2A_{\infty}conf}$. Hence we can take $\Lambda$ to be the identity natural isomorphism. Next, let $(C = (C_1 \rightrightarrows C_0), \pi, \mathbb{A})$ be an associative conformal $2$-algebra. If we apply $T$, we obtain the $2$-term $A_\infty$-conformal algebra $(\mathrm{ker}(s) \xrightarrow{t} C_0, \mu_2, \mu_3)$. Further, applying $S$, we obtain the associative conformal $2$-algebra $\big( C' = (C_0 \oplus \mathrm{ker}(s) \rightrightarrows C_0), \pi', \mathbb{A}'  \big)$. We define a functor $\Upsilon_C = \big(  (\Upsilon_C)_0, (\Upsilon_C)_1  \big) : C' \rightarrow C$ by
\begin{align*}
    (\Upsilon_C)_0 (a) = a ~~~~ \text{ and } ~~~~ (\Upsilon_C)_1 (a,m) = 1_a + m, \text{ for } a \in C_0 \text{ and } (a, m) \in C_0 \oplus \mathrm{ker}(s).
\end{align*}
We also take the identity natural isomorphism $(\mathbbm{id}_C)_{a,b} : (\Upsilon_C)_0(a)_\lambda (\Upsilon_C)_0 (b) \rightarrow (\Upsilon_C)_0 (a'_\lambda b) = (\Upsilon_C)_0 (a_\lambda b)$, for $a, b \in C_0$. It can be checked that $\big( \Upsilon_C = (  (\Upsilon_C)_0, (\Upsilon_C)_1  ), \mathbbm{id}_C  \big)$ is an isomorphism of associative conformal $2$-algebras from $C'$ to $C$. This construction yields a natural isomorphism $\Upsilon : S \circ T  \implies 1_{{\bf conf2Alg} }$. Hence the proof follows.
\end{proof}

\section{Homotopy Lie conformal algebras} \label{section-6}
In this section, we first recall Lie conformal algebras. Then we introduce and study $L_{\infty}$-conformal algebras which are the homotopy analogue of Lie conformal algebras. We show that a suitable skew-symmetrization of an $A_{\infty}$-conformal algebra gives rise to an $L_{\infty}$-conformal algebra structure.

\begin{definition}
    A {\em Lie conformal algebra} is a pair $(L,[\cdot_\lambda \cdot])$ consisting of a $\mathbb{C}[\partial]$-module $L$ with a $\mathbb{C}$-linear map (called the $\lambda$-bracket)
       $ [\cdot_\lambda \cdot] : L\otimes L \to L[\lambda], ~ x\otimes y \mapsto [x_{\lambda} y ] $
    satisfying the conformal sesquilinearity condition: 
    \begin{align*}
     [\partial x_{\lambda} y] = - \lambda [x _{\lambda}y] ~~ \text{ and } ~~ [x_{\lambda}\partial y] = (\lambda + \partial )[x_{\lambda} y]
\end{align*}
and the following identities:
     
    - skew-symmetry: $[x_{\lambda} y] = -[y_{-\lambda -\partial} x]$,
    
    - conformal Jacobi identity: $[x_{\lambda}[y_{\mu} z]] = [[x_{\lambda} y]_{\lambda + \mu} z] + [y_{\mu}[x_{\lambda} z]]$.
    \end{definition}

\medskip
    
    Let $L = \mathbb{C}[\partial]l$ be the $\mathbb{C}[\partial]$-module with one free generator $l$. Define the $\lambda$-bracket $[\cdot_\lambda \cdot ] : L \otimes L \to L[\lambda]$ by 
    \begin{align*}
        [l _{\lambda} l] = (\partial + 2\lambda)l.
    \end{align*}
    Then $(L,[\cdot_\lambda \cdot])$ is a Lie conformal algebra, called the Virasoro Lie conformal algebra and denoted by Vir.

\medskip
    
    Let $(\mathfrak{g}, [~,~])$ be a Lie algebra. The current Lie conformal algebra associated to $\mathfrak{g}$ is given by $\mathrm{Cur~}\mathfrak{g} = \mathbb{C}[\partial]\otimes\mathfrak{g}$ with the $\lambda$-bracket
    \begin{align*}
        [x_{\lambda} y] := [x,y], \mathrm{~for~ all ~} x,y \in \mathfrak{g}.
    \end{align*}
    It has been observed by D$'$Andrea and Kac \cite{andrea} that the Virasoro Lie conformal algebra and current Lie conformal algebras $\mathrm{Cur~}\mathfrak{g}$ (where $\mathfrak{g}$ is finite-dimensional simple Lie algebra) exhaust all finite simple Lie conformal algebras.

\begin{remark} \label{skew-rem}(\cite{bakalov-kac-voronov}) Let $(A, \cdot_\lambda \cdot)$ be an associative conformal algebra. Then the $\mathbb{C}[\partial]$-module $A$ equipped with the $\lambda$-bracket $[a_\lambda b] := a_\lambda b - b_{-\partial - \lambda} a$, for $a, b \in A$, is a Lie conformal algebra. This is called the skew-symmetrization of the associative conformal algebra $(A, \cdot_\lambda \cdot)$.
\end{remark}
    
    \begin{definition}
        Let $(L,[\cdot_\lambda \cdot])$ be a Lie conformal algebra. A {\em conformal $L$-module} is a pair $(M, \cdot_\lambda \cdot)$ in which $M$ is a $\mathbb{C}[\partial]$-module and 
        \begin{align*}
            \cdot_\lambda \cdot : L \otimes M \to M[\lambda], ~ x \otimes v \mapsto x_{\lambda} v
        \end{align*}
        is a $\mathbb{C}$-linear map satisfying the following properties: for all $x,y \in L$ and $v \in M$,
        \begin{align*}
           & (\partial x)_{\lambda} v = -\lambda (x_{\lambda} v), ~~~ x_{\lambda}(\partial v) = (\lambda + \partial )x_{\lambda} v,\\&
           x_{\lambda}(y_{\mu}v)-y_{\mu}(x_{\lambda} v) = [x_{\lambda} y ]_{\lambda +\mu}v.
        \end{align*}
        \end{definition}
        It follows that any Lie conformal algebra is a conformal module over itself, called the adjoint module.

\medskip
        
In the following, we recall the graded Lie algebra whose Maurer-Cartan elements correspond to Lie conformal algebra structures on a given $\mathbb{C}[\partial]$-module \cite{sole-kac2,wu} (see also \cite{yuan-liu}). Let $L$ be a $\mathbb{C}[\partial]$-module. A conformal sesquilinear map 
\begin{align*}
\varphi: L^{\otimes n} \to L[\lambda_{1},\ldots,\lambda_{n-1}], ~x_1 \otimes \cdots \otimes x_n \mapsto \varphi_{\lambda_1, \ldots, \lambda_{n-1}} (x_1, \ldots, x_n)
\end{align*}
is said to be {\em skew-symmetric} \cite{sole-kac} if it is skew-symmetric with respect to simultaneous permutations of the
$x_i$'s and the $\lambda_i$'s in the sense that, for any permutation $\sigma \in \mathbb{S}_n$, we have
\begin{align}\label{skew-define}
    \varphi_{\lambda_{1},\ldots,\lambda_{n-1}}(x_{1},\ldots,x_{n}) = \mathrm{sgn} (\sigma)~ \varphi_{\lambda_{\sigma (1)},\ldots,\lambda_{\sigma(n-1)}}(x_{\sigma(1)},\ldots,x_{\sigma (n-1)},x_{\sigma (n)}) \big|_{\lambda_n \mapsto \lambda_n^\dagger}.
\end{align}
for all $x_{1},\ldots,x_{n} \in L$. 
The notation on the right-hand side of (\ref{skew-define}) simply means that $\lambda_n$ is replaced by $\lambda_n^\dagger = - \sum_{i=1}^{n-1} \lambda_i - \partial$, if it occurs and $\partial$ is moved to the left.
We denote the set of all such skew-symmetric conformal sesquilinear maps by $\mathrm{Hom}_{cs}(\wedge ^{n}L, L[\lambda_{1},\ldots,\lambda_{n-1}])$. Define a graded vector space 
\begin{align*}
\mathfrak{g} = {\oplus}_{n \geq 1}\mathrm{Hom}_{cs}(\wedge ^{n}L,L[\lambda_{1},\ldots,\lambda_{n-1}]).
\end{align*}
For any $\varphi \in \mathrm{Hom}_{cs}(\wedge ^{n}L,L[\lambda_{1},\ldots,\lambda_{n-1}])$ and $\psi \in \mathrm{Hom}_{cs}(\wedge^{m}L,L[\lambda_{1},\ldots,\lambda_{m-1}])$, we define a bracket $[\varphi, \psi]_{\mathrm{CNR}} \in \mathrm{Hom}_{cs}(\wedge ^{m+n-1}L,L[\lambda_{1},\ldots,\lambda_{m+n-2}])$ by 
\begin{align*}
    [\varphi, \psi]_{\mathrm{CNR}} = \varphi \diamond \psi -(-)^{(m-1)(n-1)}\psi \diamond \varphi, \mathrm{~where}
\end{align*}
\begin{align}\label{cnr-p}
(\varphi \diamond \psi &)_{\lambda_{1},\ldots,\lambda_{m+n-2}}(x_{1},\ldots,x_{m+n-1})
=\sum_{\sigma \in \mathrm{Sh}(m,n-1)} \mathrm{sgn}{(\sigma)}~ \\ &\varphi_{ \lambda_\sigma^\dagger,\lambda_{\sigma(m+1)},\ldots,\lambda_{\sigma(m+n-2)}}(\psi_{\lambda_{\sigma(1)},\ldots,\lambda_{\sigma(m-1)}}(x_{\sigma(1)},\ldots,x_{\sigma(m)}),x_{\sigma(m+1)},\ldots, x_{\sigma(m+n-1)} ), \nonumber
\end{align}   
for $x_1, \ldots, x_{m+n-1} \in L$. Note that $\sigma \in \mathrm{Sh} (m,n-1)$ implies that $\sigma (m) = m+n-1$ or $\sigma (m+n-1) = m+n-1$. If $\sigma (m) = m+n-1$, we use the notation $\lambda_\sigma^\dagger = - \partial - (  \lambda_{\sigma (m+1)} + \cdots + \lambda_{\sigma (m+n-1)})$ and if $\sigma (m+n-1) = m+n-1$, we use the notation $\lambda_\sigma^\dagger = \lambda_{\sigma (1)} + \cdots + \lambda_{\sigma (m)}$ in the above expression (\ref{cnr-p}).
With this, it is not hard to verify that the product $\diamond$ makes the graded vector space $\mathfrak{g}$ into a degree $-1$ graded pre-Lie algebra. Hence the bracket $[~,~]_{\mathrm{CNR}}$ is a degree $-1$ graded Lie bracket on $\mathfrak{g}$. In other words, the shifted graded vector space $\mathfrak{g}[1] = {\oplus}_{n \geq 0}\mathrm{Hom}_{cs}(\wedge ^{n+1}L,L[\lambda_{1},\ldots,\lambda_{n}])$ is a graded Lie algebra. For $\pi \in \mathrm{Hom}_{cs}(\wedge ^{2}L,L[\lambda])$, we have 
\begin{align*}
    ([\pi, \pi]_{\mathrm{CNR}})_{\lambda, \mu}(x,y,z) = 2(\pi_{\lambda+ \mu}(\pi_{\lambda}(x,y),z)-\pi_{- \partial - \mu}(\pi_{\lambda}(x,z),y)+ \pi_{ - \partial - \lambda}(\pi_{\mu}(y,z),x)),
\end{align*}
for $x,y,z \in L$. This shows that $\pi$ is a Maurer-Cartan element in the graded Lie algebra $(\mathfrak{g}[1],[~,~]_{\mathrm{CNR}})$ if and only if the $\lambda$-bracket 
$[\cdot_\lambda \cdot] : L \otimes L \to L[\lambda], ~ [x _{\lambda}y] := \pi_{\lambda}(x,y)$ defines a Lie conformal algebra structure on $L$. This characterization allows us to express the known coboundary operator of a Lie conformal algebra (\cite{sole-kac}) in terms of the graded Lie bracket $[~,~]_\mathrm{CNR}$.

Let $(L, [\cdot_\lambda \cdot ])$ be a Lie conformal algebra with the Maurer-Cartan element $\pi \in \mathrm{Hom}_{cs} (\wedge^2L, L [\lambda])$ in the graded Lie algebra $(\mathfrak{g}[1], [~,~]_\mathrm{CNR})$. For each $n \geq 0$, we define the $n$-th cochain group $C^{n}(L,L)$ by 
\begin{equation*}
C^{n}(L,L)=
    \begin{cases}
        L/{\partial L} & \text{if  } n = 0,\\
        \mathrm{Hom}_{cs}(\wedge^{ n}L, L[\lambda_{1},\ldots,\lambda_{n-1}]) & \text{if  }  n \geq 1.
    \end{cases}
\end{equation*} 
Then there is a map $\delta_\pi : C^{n}(L,L) \to C^{n+1}(L,L)$, for $n \geq 0$, given by 
\begin{align*}
    \delta_\pi (x + \partial L )(y) =~& [y_{-\lambda - \partial} x] \big|_{\lambda = 0}, \mathrm{~for ~} x+\partial L \in {L}/{\partial L},\\
    \delta_\pi (\varphi) =~& (-1)^{n-1} [\pi, \varphi]_{\mathrm{CNR}}, \mathrm{~for~} \varphi \in C^{n \geq 1}(L,L).
\end{align*}
%Explicitly, the map $\delta$ is given by 
%\begin{align*}
%   & (\delta \varphi)_{\lambda_{1},\ldots,\lambda_{n}}(x_{1},\ldots,x_{n+1})=\sum_{i = }^{n+1}(-1)^{i+1}[{x_{i}}_{\lambda_{i}} \varphi_{\lambda_{1},\ldots,\hat{\lambda_{i}},\ldots,\lambda_{n}}(x_{1},\ldots,\hat{x_{i}},\ldots,x_{n+1})]\\&
%   +\sum_{1 \leq i < j \leq n+1 }(-1)^{i+j}\varphi_{\lambda_{i}+\lambda_{j},\ldots,\hat{\lambda_{i}},\ldots,\hat{\lambda_{j}},\ldots,\lambda_{n}}([{x_{i}}_{\lambda_{i}} x_{j}],x_{1},\ldots,\hat{x_{i}},\ldots,\hat{x_{j}},\ldots,x_{n+1}),
%\end{align*}
%for $ \varphi \in C^{n}(L,L)$ and $x_{1}, \ldots , x_{n+1} \in L$. 
Since $[\pi,\pi]_{\mathrm{CNR}} = 0$, it follows that $(\delta_\pi)^{2} = 0$. In other words, $\{C^{\bullet}(L,L),\delta_\pi \}$ is a cochain complex. The corresponding cohomology groups are called the {\em cohomology} of the Lie conformal algebra $(L,[\cdot_\lambda \cdot])$ with coefficients in the adjoint module.

The above construction can be easily generalized when we have an arbitrary conformal module. Let $(L,[\cdot_\lambda \cdot])$ be  a Lie conformal algebra and $(M, \cdot_\lambda \cdot)$ be any conformal module over it. Consider the $\mathbb{C}[\partial]$-module $L \oplus M$ with the $\lambda$-bracket $[\cdot_\lambda \cdot]^\ltimes : (L \oplus M) \otimes (L \oplus M) \rightarrow (L \oplus M)[\lambda]$ given by
\begin{align*}
    [(x, u)_\lambda (y, v)]^\ltimes := \big(   [x_\lambda y], x_\lambda v - y_{- \partial - \lambda} u \big), \text{ for } (x, u), (y, v) \in L \oplus M.
\end{align*}
This is a Lie conformal algebra, called the {\em semidirect product}. Let $\pi_\ltimes \in \mathrm{Hom}_{cs} \big(\wedge^2 (L \oplus M), (L \oplus M)[\lambda] \big)$ be the Maurer-Cartan element corresponding to the semidirect product. Then $\pi_\ltimes$ yields the coboundary operator $\delta_{\pi_\lambda}$ of the cochain complex $\{ C^\bullet (L \oplus M, L \oplus M), \delta_{\pi_\lambda} \}$ defining the cohomology of the semidirect product Lie conformal algebra with coefficients in the adjoint module. For each $n \geq 0$, we consider a space $C^n(L, M)$ by
\begin{align*}
C^{0}(L,M) = {M}/{\partial M} \quad \mathrm{and} \quad C^{n \geq 1}(L,M) = \mathrm{Hom}_{cs}(\wedge^{ n}L, M[\lambda_{1},\ldots,\lambda_{n-1}]).
\end{align*}
It is easy to see that the coboundary map $\delta_{\pi_\ltimes}$ restricts to a coboundary map (denoted by $\delta_\pi$) $\delta_\pi : C^{n}(L,M) \rightarrow C^{n+1}(L,M)$. In other words, $\{ C^\bullet (L, M), \delta_\pi \}$ is a cochain complex. The map $\delta_\pi$ is explicitly  given by $\delta_\pi (v + \partial M)(x) = (x_{-\lambda - \partial } v) \big|_{\lambda = 0}$, for $ v + \partial M  \in C^{0}(L,M)$, and
\begin{align*}
   & (\delta_\pi \varphi)_{\lambda_{1},\ldots,\lambda_{n}}(x_{1},\ldots,x_{n+1}) \\
   &=\sum_{i = 1}^{n}(-1)^{i+1}~{x_{i}}_{\lambda_{i}} \varphi_{\lambda_{1},\ldots,\widehat{\lambda_{i}},\ldots,\lambda_{n}}(x_{1},\ldots,\widehat{x_{i}},\ldots,x_{n+1}) \\
   &+ (-1)^n ~ {x_{n+1}}_{ - \partial - \sum_{l=1}^n \lambda_l } \varphi_{\lambda_1, \ldots, \lambda_{n-1}} (x_1, \ldots, x_n) \\
   &+\sum_{1 \leq i < j \leq n }(-1)^{i+j} ~ \varphi_{- \partial - \sum_{\substack{l=1\\ l \neq i, j}}^n \lambda_l , \lambda_1,\ldots,\widehat{\lambda_{i}},\ldots,\widehat{\lambda_{j}},\ldots,\lambda_{n}}([{x_{i}}_{\lambda_{i}} x_{j}],x_{1},\ldots,\widehat{x_{i}},\ldots,\widehat{x_{j}},\ldots,x_{n+1}) \\
  & + \sum_{i=1}^n (-1)^{i+n+1} ~\varphi_{ - \partial - \sum_{ \substack{l=1 \\ l \neq i}}^n \lambda_l   , \lambda_1, \ldots, \widehat{\lambda_i}, \ldots, \lambda_{n-1}} \big( [{x_i}_{\lambda_i} x_{n+1}], x_1, \ldots, \widehat{x_i}, \ldots, x_n  \big),
\end{align*}
for $ \varphi \in C^{n}(L,M)$ and $x_{1}, \ldots , x_{n+1} \in L$. This is precisely the one given in \cite{sole-kac}. The cohomology groups of the cochain complex $\{C^{\bullet}(L,M),\delta_\pi \}$ are called the {\em cohomology} of the Lie conformal algebra $L$ with coefficients in the conformal $L$-module $M$.

\medskip

In the following, we introduce $L_{\infty}$-conformal algebras and find a relation with $A_{\infty}$-conformal algebras introduced earlier. We first recall the notion of $L_{\infty}$-algebras.
\begin{definition} (\cite{lada-markl,lada-stasheff})
    An {\em $L_{\infty}$-algebra} is a pair $(\mathcal{L}, \{l_{k}\}_{k \geq 1})$ consisting of a graded vector space $\mathcal{L} = \oplus_{i \in \mathbb{Z}} \mathcal{L}_{i}$ equipped with a collection of graded linear maps $\{l_{k} : \mathcal{L}^{\otimes k} \to \mathcal{L}\}_{k \geq 1}$ with $\mathrm{deg}(l_{k}) = k-2$ for $k \geq 1$, that satisfies the following conditions:
    
        - each $l_{k}$ is graded skew-symmetric in the sense that 
        \begin{align*}
            l_{k}(x_{1},\ldots,x_{k}) = \mathrm{sgn}(\sigma) \epsilon(\sigma) ~ l_{k}(x_{\sigma(1)},\ldots,x_{\sigma(k)}), \mathrm{~for ~all ~} \sigma \in \mathbb{S}_{k},
        \end{align*}
        where $\epsilon(\sigma)$ is the Koszul sign that appears in the graded context, 
        
        - higher Jacobi identities: for $n \geq 1$ and homogeneous elements $x_{1},\ldots,x_{n} \in \mathcal{L}$, 
        \begin{align*}
            \sum_{p+q = n+1}\sum_{\sigma \in \mathrm{Sh}(q,n-q)}\mathrm{sgn}(\sigma) \epsilon(\sigma)(-1)^{q(p-1)} ~l_{p}(l_{q}(x_{\sigma(1)},\ldots,x_{\sigma(q)}),x_{\sigma(q+1)},\ldots,x_{\sigma(n)}) = 0.
        \end{align*}
\end{definition}

\medskip

Any Lie algebras, graded Lie algebras and differential graded Lie algebras are particular cases of $L_{\infty}$-algebras. More precisely, a differential graded Lie algebra $(\mathfrak{g},[~,~],d)$ is an $L_{\infty}$-algebra $(\mathfrak{g},\{l_{k}\}_{k \geq 1})$, where $l_{1} = d,~ l_{2} = [~,~] $ and $l_{k}=0$ for $k \geq 3$.
\begin{definition}
    An {\em $L_{\infty}$-conformal algebra} (also called a {\em strongly homotopy Lie conformal algebra}) is a graded $\mathbb{C}[\partial]$-module $\mathcal{L} = \oplus_{i \in \mathbb{Z}} \mathcal{L}_{i}$ with a collection $\{l_{k} : \mathcal{L}^{\otimes k} \to \mathcal{L}[\lambda_{1},\ldots, \lambda_{k-1}]\}_{k \geq 1}$ of graded $\mathbb{C}$-linear maps (called the $\lambda$-multiplications) with $\mathrm{deg}(l_{k}) = k-2$ for $k \geq 1$, subject to the following conditions:
    
    - each $l_{k}$ is conformal sesquilinear,
    
    - each $l_{k}$ is graded skew-symmetric in the sense that 
    %\begin{align*}
     %   (l_{k})_{\lambda_{1},\dots,\lambda_{k-1}}(x_{1},\ldots,x_{k}) = -(-1)^{|x_{i}|x_{i+1}|}(l_{k})_{\lambda_{1},\ldots,-\lambda_{i}-\partial,\ldots,\lambda_{k-1}}(x_{1},\ldots,x_{i+1},x_{i},\ldots,x_{k}),
    %\end{align*}

\begin{align*}
    (l_k)_{\lambda_{1},\ldots,\lambda_{k-1}}(x_{1},\ldots,x_{k}) = \mathrm{sgn} (\sigma) \epsilon (\sigma)~ (l_k)_{\lambda_{\sigma (1)},\ldots,\lambda_{\sigma(k-1)}}(x_{\sigma(1)},\ldots,x_{\sigma (k-1)},x_{\sigma (k)}) \big|_{\lambda_k \mapsto \lambda_k^\dagger}, \text{ for } \sigma \in \mathbb{S}_k,
\end{align*}
    
    - higher conformal Jacobi idntities: for any $n \geq 1$ and homogeneous $x_1, \ldots, x_n \in \mathcal{L}$,
     \begin{align}
            \sum_{p+q = n+1} & \sum_{\sigma \in \mathrm{Sh}(q,n-q)}\mathrm{sgn}(\sigma) \epsilon(\sigma)(-1)^{q(p-1)} \\
           & (l_{p})_{  \lambda^\dagger_\sigma ,  \lambda_{\sigma (q+1)}, \ldots, \lambda_{\sigma (n-1)}  } \big(  (l_{q})_{ \lambda_{\sigma (1)}, \ldots, \lambda_{\sigma (q-1)}} (x_{\sigma(1)},\ldots,x_{\sigma(q)}),x_{\sigma(q+1)},\ldots,x_{\sigma(n)} \big) = 0. \nonumber
        \end{align}
        Note that $\sigma \in \mathrm{Sh}(q, n-q)$ implies that either $\sigma (q) = n$ or $\sigma (n) = n$. When $\sigma (q) = n$, we use the notation $\lambda_\sigma^\dagger =  - \partial - (\lambda_{\sigma (q+1)} + \cdots + \lambda_{\sigma (n)})$ and when $\sigma (n) = n$, we use the notation $\lambda_\sigma^\dagger = \lambda_{\sigma (1)} + \cdots + \lambda_{\sigma (q)}$ in the above identities.
    \end{definition}

    \medskip

Lie conformal algebras and (differential) graded Lie conformal algebras are examples of $L_{\infty}$-conformal algebras. The notion of $2$-term $L_\infty$-conformal algebras considered in \cite{tao} and examples therein are particular case of $L_\infty$-conformal algebras defined above. Similar to $A_\infty$-conformal algebras, the concept of $L_{\infty}$-conformal algebras has a simple interpretation when we shift the degree. We first consider the following definition.
\begin{definition}
    An {\em $L_{\infty}[1]$-conformal algebera} is a pair $(\mathcal{W},\{\varrho_{k}\}_{k \geq 1})$ that consists of a graded $\mathbb{C}[\partial]$-module $\mathcal{W} = \oplus_{i \in \mathbb{Z}}\mathcal{W}_{i}$ with a collection $\{\varrho_{k} : \mathcal{W}^{\otimes k} \to \mathcal{W}[\lambda_{1}, \ldots , \lambda_{k-1}]\}_{k \geq 1}$ of graded $\mathbb{C}$-linear maps with deg$(\varrho_{k}) = -1$ subject to the following conditions:
    
    - each $\varrho_{k}$ is conformal sesquilinear,
    
    - each $\varrho_{k}$ is graded symmetric, i.e. 
    \begin{align}
    (\varrho_k)_{\lambda_{1},\ldots,\lambda_{k-1}}(x_{1},\ldots,x_{k}) = \epsilon (\sigma)~ (\varrho_k)_{\lambda_{\sigma (1)},\ldots,\lambda_{\sigma(k-1)}}(x_{\sigma(1)},\ldots,x_{\sigma (k-1)},x_{\sigma (k)}) \big|_{\lambda_k \mapsto \lambda_k^\dagger}, \text{ for } \sigma \in \mathbb{S}_k,
\end{align}
        
    - shifted higher conformal Jacobi identities: for any $n \geq 1 $ and homogeneous elements $ w_{1}, \ldots , w_{n} \in \mathcal{W}$,
\begin{align*}
            \sum_{p+q = n+1} & \sum_{\sigma \in \mathrm{Sh}(q,n-q)} \epsilon(\sigma) \\
           & (\varrho_{p})_{  \lambda_\sigma^\dagger ,  \lambda_{\sigma (q+1)}, \ldots, \lambda_{\sigma (n-1)}  } \big(  (\varrho_{q})_{ \lambda_{\sigma (1)}, \ldots, \lambda_{\sigma (q-1)}} (x_{\sigma(1)},\ldots,x_{\sigma(q)}),x_{\sigma(q+1)},\ldots,x_{\sigma(n)} \big) = 0.
        \end{align*}
\end{definition}

\medskip

The concept of $L_\infty [1]$-conformal algebra is the conformal analogue of an $L_\infty[1]$-algebra considered by Lada and Markl \cite{lada-markl}. The following result is a straightforward generalization of a result proved in \cite{lada-markl}.

\begin{proposition}
    Let $\mathcal{L}$ be a graded $\mathbb{C}[\partial]$-module. Then an $L_{\infty}$-conformal algebra structure on $\mathcal{L}$ is equivalent to having an $L_{\infty}[1]$-conformal algebra structure on the graded $\mathbb{C}[\partial]$-module $\mathcal{L}[-1]$. Explicitly, $(\mathcal{L}, \{l_{k}\}_{k \geq 1})$ is an $L_{\infty}$-conformal algebra if and only if $(\mathcal{L}[-1], \{\varrho_{k}\}_{k \geq 1})$ is an $L_{\infty}[1]$-conformal algebra, where 
    \begin{align*}
    \varrho_{k} = (-1)^{\frac{k(k-1)}{2}} ~s \circ l_{k} \circ (s^{-1})^{\otimes k}, ~\text{ for } k \geq 1.
    \end{align*}
\end{proposition}

\medskip

Let $\mathcal{W} = \oplus_{i \in \mathbb{Z}}\mathcal{W}_{i}$ be a graded $\mathbb{C}[\partial]$-module. Let $n \in \mathbb{Z}$ be an integer. For each $k \geq 1$, let 
\begin{align*}
\mathrm{Hom}^{n}_{cs}(S^{k}\mathcal{W}, \mathcal{W}[\lambda_{1},\ldots,\lambda_{k-1}])
\end{align*}
be the set of all conformal sesquilinear maps from $\mathcal{W}^{\otimes k}$ to $\mathcal{W}[\lambda_{1},\ldots,\lambda_{k-1}]$ that are graded symmetric and of degree $n$. That is,
\begin{align*}
   \mathrm{Hom}^{n}_{cs}(S^{k}\mathcal{W}, \mathcal{W}[\lambda_{1},\ldots,\lambda_{k-1}]) = \{\varphi \in \mathrm{Hom}^{n}_{cs}(\mathcal{W}^{\otimes k}, \mathcal{W}[\lambda_{1},\ldots,\lambda_{k-1}]) ~| ~ \varphi \mathrm{~is ~graded~symmetric}\} .
\end{align*}
Define $C^{\bullet}_{cs}(S\mathcal{W}) = \oplus_{n \in \mathbb{Z}}C^{n}_{cs}(S\mathcal{W})$, where
$C^{n}_{cs}(S\mathcal{W}) = \oplus_{k \geq 1}\mathrm{Hom}^{n}_{cs}(S^{k}\mathcal{W}, \mathcal{W}[\lambda_{1},\ldots,\lambda_{k-1}]).$
Thus, an element $\varrho \in C^{n}_{cs}(S\mathcal{W})$ is of the form $\varrho = \sum_{k \geq 1}\varrho_{k}$, where $\varrho_{k} \in \mathrm{Hom}^{n}_{cs}(S^{k}\mathcal{W}, \mathcal{W}[\lambda_{1},\ldots,\lambda_{k-1}])$ for all $k \geq 1$. For $\varrho = \sum_{k \geq 1}\varrho_{k} \in C^{n}_{cs}(S\mathcal{W})$ and $\tau = \sum_{l \geq 1} \tau_{l} \in C^{m}_{cs}(S\mathcal{W})$, we define
\begin{align}\label{new-gla}
    \{\![\sum_{k \geq 1}\varrho_{k}, \sum_{l \geq 1}\tau_{l}]\!\} = \sum_{p \geq 1} \sum_{k +l = p+1} (\varrho_{k} ~\overline{\diamond}~ \tau_{l} - (-1)^{mn}~\tau_{l} ~\overline{\diamond}~ \varrho_{k}),
\end{align}
where 
\begin{align*}
    (\varrho_{k} &~\overline{\diamond}~  \tau_{l})_{\lambda_{1}, \ldots, \lambda_{p-1}}(w_{1},\ldots,w_{p}) \\
    &= \sum_{\sigma \in \mathrm{Sh} (l, k-1)} \epsilon (\sigma)~ (\varrho_{k})_{  \lambda_\sigma^\dagger ,  \lambda_{\sigma (l+1)}, \ldots, \lambda_{\sigma (p-1)}  } \big( (\tau_{l})_{\lambda_{\sigma (1)}, \ldots, \lambda_{\sigma (l-1)}} (w_{\sigma(1)},\ldots,w_{\sigma(l)}),w_{\sigma(l+1)},\ldots,w_{\sigma(p)} \big).
\end{align*}
Then $(C^{\bullet}_{cs}(S\mathcal{W}),\{\![ ~, ~]\!\})$ is a graded Lie algebra. Moreover, we have the following result.
\begin{thm}
    Let $\mathcal{L}$ be a graded $\mathbb{C}[\partial]$-module. Then an $L_{\infty}$-conformal algebra structure on $\mathcal{L}$ is equivalent to the existence of a Maurer-Cartan element in the graded Lie algebra  $(C^{\bullet}_{cs}(S\mathcal{W}),\{\![ ~, ~]\!\})$, where $\mathcal{W} = \mathcal{L}[-1]$.
\end{thm}
\begin{proof}
We have seen earlier that an $L_\infty$-conformal algebra structure on $\mathcal{L}$ is equivalent to an $\mathcal{L}_\infty [1]$-conformal algebra structure on $\mathcal{W} = \mathcal{L}[-1].$ Next, it follows from (\ref{new-gla}) that an $L_\infty [1]$-conformal algebra structure on $\mathcal{W}$ is equivalent to a Maurer-Cartan element in the graded Lie algebra $(C^\bullet_{cs} (S \mathcal{W}),  \{\![ ~, ~]\!\} )$. Hence the result follows. 
\end{proof}

In \cite{lada-markl} Lada and Markl showed that the skew-symmetrization of an $A_\infty$-algebra gives rise to an $L_\infty$-algebra. On the other hand, an associative conformal algebra gives a Lie conformal algebra by a suitable skew-symmetrization (cf. Remark \ref{skew-rem}). Unifying both the above results, we have the following.

\begin{thm}
    Let $(\mathcal{A},\{\mu_{k}\}_{k \geq 1})$ be an $A_{\infty}$-conformal algebra. Then the underlying graded $\mathbb{C}[\partial]$-module $\mathcal{A}$ carries an $L_{\infty}$-conformal algebra structure with the $\lambda$-multiplications $\{l_{k} : \mathcal{A}^{\otimes k} \to \mathcal{A}[\lambda_{1}, \ldots, \lambda_{k-1}]\}_{k \geq 1}$ given by 
    \begin{align*}
    (l_{k})_{\lambda_{1}, \ldots, \lambda_{k-1}}(a_{1},\ldots,a_{k}) = \sum_{\sigma \in \mathbb{S}_{k}} \mathrm{sgn} (\sigma) \epsilon(\sigma) ~(\mu_{k})_{\lambda_{\sigma (1)}, \ldots, \lambda_{\sigma (k)}} 
 (a_{\sigma(1)},\ldots,a_{\sigma(k)}) \big|_{\lambda_k \mapsto \lambda_k^\dagger},\mathrm{~for~ } a_{1}, \ldots, a_{k} \in \mathcal{A}.
    \end{align*}
\end{thm}

\begin{proof}
    Since $(\mathcal{A}, \{ \mu_k \}_{k \geq 1})$ is an $A_\infty$-conformal algebra, it follows from Proposition \ref{proposition 3.6} that $(\mathcal{A}[-1], \{ \rho_k \}_{k \geq 1})$ is an $A_\infty [1]$-conformal algebra. Hence $\rho = \sum_{k \geq 1} \rho_k \in C^{-1}_{cs} (\mathcal{A}[-1])$ is a Maurer-Cartan element in the graded Lie algebra $(C^\bullet_{cs} (\mathcal{A}[-1]), \llbracket ~, ~ \rrbracket)$.

    On the other hand, a symmetrization of conformal sesquilinear maps yields a morphism $\mathfrak{S} : C^\bullet_{cs} (\mathcal{A}[-1]) \rightarrow C^\bullet_{cs} (S \mathcal{A}[-1])$ of graded vector spaces that preserves the corresponding graded Lie brackets. Further, we have $\mathfrak{S} (\rho) = \varrho$, where $\varrho = \sum_{k \geq 1} \varrho_k$ with $\varrho_k = (-1)^{\frac{k(k-1)}{2}} ~ s \circ l_k \circ (s^{-1})^{\otimes k}$, for $k \geq 1$. This shows that
    \begin{align*}
        \{ \! [ \varrho, \varrho ] \! \} = \{ \! [ \mathfrak{S}( \rho), \mathfrak{S} (\rho) ] \! \} = \mathfrak{S} \llbracket \rho, \rho \rrbracket = 0.
    \end{align*}
    Hence $\varrho = \sum_{k \geq 1} \varrho_k$ defines an $L_\infty[1]$-conformal algebra structure on the graded vector space $\mathcal{A}[-1]$. In other words, $\{ l_k \}_{k \geq 1}$ gives rise to an $L_\infty$-conformal algebra structure on $\mathcal{A}$.
\end{proof}

\medskip

\noindent {\bf Conflict of interest statement.} There is no conflict of interest.

\medskip

\noindent {\bf Data Availability Statement.} Data sharing is not applicable to this article as no new data were created or analyzed in this study.

\medskip

\noindent {\bf Acknowledgements.} 
Both authors thank Jiefeng Liu for his comments on the earlier version of the paper. Anupam Sahoo would like to thank CSIR, Government of India for funding the PhD fellowship. Both authors are greatful to Department of Mathematics, IIT Kharagpur for providing the beautiful academic atmosphere where the research has been carried out.

\end{document}